\documentclass[12pt]{amsart}
\usepackage{amssymb,latexsym,eufrak,amsmath,amscd, graphicx}

\usepackage{amsfonts}
\usepackage{amscd,xypic}

\renewcommand{\phi}{\varphi}
\renewcommand{\epsilon}{\varepsilon}
\renewcommand{\theta}{\vartheta}

\def\ZZ{{\mathbf Z}}

\def\CC{{\mathbf C}}
\def\AAA{{\mathbf A}}
\def\RR{{\mathbf R}}
\def\QQ{{\mathbf Q}}

\def\cH{\mathcal{H}}
\def\cI{\mathcal{I}}
\def\cJ{\mathcal{J}}

\def\cA{\mathcal{A}}
\def\cF{\mathcal{F}}
\def\cG{\mathcal{G}}

\def\cO{\mathcal{O}}
\def\cZ{\mathcal{Z}}

\def\cZ{\mathcal{Z}}

\def\fra{\mathfrak{a}}
\def\ofra{\overline{\fra}}
\def\ofrm{\overline{\mathfrak m}}
\def\ofrb{\overline{\frb}}

\def\frb{\mathfrak{b}}

\def\frm{\mathfrak{m}}
\def\frn{\mathfrak{n}}
\def\frp{\mathfrak{p}}
\def\frq{\mathfrak{q}}

\def\Z{{\mathbf Z}}

\def\C{{\mathbf C}}

\def\R{{\mathbf R}}
\def\Q{{\mathbf Q}}

\def\O{\mathcal{O}}

\def\frA{\mathfrak{A}}
\def\frB{\mathfrak{B}}

\DeclareMathOperator{\codim}{codim} 
\DeclareMathOperator{\Hom}{Hom}

 \DeclareMathOperator{\Spec}{Spec}
 
 \DeclareMathOperator{\lct}{lct}
 
 \DeclareMathOperator{\ord}{ord}

\DeclareMathOperator{\reg} {reg}

\def\.{\cdot}
\def\~{\widetilde}
\def\^{\widehat}

\def\o{\circ}
\def\ov{\overline}
\def\rat{\dashrightarrow}

\newcommand{\llbracket}{[\negthinspace[}
\newcommand{\rrbracket}{]\negthinspace]}

\newtheorem{lemma}{Lemma}[section]
\newtheorem{theorem}[lemma]{Theorem}
\newtheorem{corollary}[lemma]{Corollary}
\newtheorem{proposition}[lemma]{Proposition}

\theoremstyle{definition}

\newtheorem{remark}[lemma]{Remark}

\newtheorem{example}[lemma]{Example}

\theoremstyle{remark}
\newtheorem*{remark*}{Remark}
\newtheorem*{note*}{Note}

\title{Log canonical thresholds on varieties with bounded singularities}

\author[T.~de Fernex]{Tommaso de Fernex}
\address{Department of Mathematics, University of Utah, 155 South 1400 East,
Salt Lake City, UT 48112-0090, USA} \email{{\tt
defernex@math.utah.edu}}

\author[L. Ein]{Lawrence~Ein}
\address{Department of Mathematics, University of
Illinois at Chicago, 851 South Morgan Street (M/C 249),
Chicago, IL 60607-7045, USA}
\email{{\tt ein@math.uic.edu}}

\author[M. Musta\c{t}\u{a}]{Mircea~Musta\c{t}\u{a}}
\address{Department of Mathematics, University of Michigan,
Ann Arbor, MI 48109, USA}
\email{{\tt mmustata@umich.edu}}

\thanks{2010\,\emph{Mathematics Subject Classification}.
 Primary 14E15; Secondary 14B05, 14E30.
\newline
The first author was partially supported by NSF
CAREER grant DMS-0847059,
the second author
  was partially supported by NSF grant DMS-0700774,
  and the third author was partially supported by
  NSF grant DMS-0758454, and
  by a Packard Fellowship}

\begin{document}

\begin{abstract}
We consider pairs $(X,\frA)$, where $X$ is a variety with klt singularities and $\frA$
is a formal product of ideals on $X$ with exponents in a fixed set that satisfies the Descending
Chain Condition. We also assume that $X$ has (formally)
bounded singularities, in the sense that it is,
formally locally, a subvariety in a fixed affine space defined by equations
of bounded degree. We prove in this context a conjecture of Shokurov, predicting that the set
of log canonical thresholds for such pairs satisfies the Ascending Chain Condition.
\end{abstract}

\maketitle

\markboth{T. DE FERNEX, L.~EIN AND M.~MUSTA\c{T}\u{A}}
{LOG CANONICAL THRESHOLDS ON VARIETIES WITH BOUNDED SINGULARITIES}

\section{Introduction}

The log canonical threshold is a fundamental invariant in birational geometry. 
It is attached to a divisor with real coefficients on a variety with mild singularities.
An outstanding conjecture due to Shokurov \cite{Sho} 
predicts that in any fixed dimension, if the coefficients of the divisors are taken 
in any given set of positive real numbers satisfying the
descending chain condition (DCC), then the set of all
possible log canonical thresholds satisfies the ascending chain condition 
(ACC).\footnote{A set of real numbers satisfies DCC (respectively, ACC) if
it does not contain any infinite sequence 
that is strictly decreasing (respectively, strictly increasing).
For short, such a set will be called a {\it DCC set} (respectively, {\it ACC set}).}

This conjecture has attracted considerable
attention due to its implications to the Termination of Flips Conjecture. 
More precisely, Birkar showed in \cite{Birkar} the following: if Shokurov's conjecture is known in dimension $n$,
and if the log Minimal Model Program is known in dimension $(n-1)$, then there are no infinite sequences of flips in dimension $n$ for pairs of non-negative log Kodaira dimension. We note that
due to the results in \cite{BCHM}, Termination of Flips is the remaining piece in order to establish
the log Minimal Model Program in arbitrary dimension. There is another outstanding open problem in the area, the Abundance Conjecture, but the circle of ideas we are discussing does not have anything to say in that direction.

Shokurov's conjecture was proved in the case of smooth (and, more generally, locally complete 
intersection)
ambient varieties in \cite{dFEM}, building on work from \cite{dFM} and \cite{Kol1}.
In this note we deal with the more general case of varieties that have bounded singularities,
in a sense to be explained below.

Let $k$ be an algebraically closed field of characteristic zero.
We assume that our ambient varieties are defined over $k$, and are normal and 
$\QQ$-Gorenstein. Let $X$ be any such variety. 
Instead of dealing with $\RR$-divisors, 
we work in the more general setting of \emph{$\RR$-ideals} on $X$: these are formal products 
$\frA=\fra_1^{q_1}\cdots\fra_r^{q_r}$, where the $\fra_i$ are nonzero ideal sheaves 
and the $q_i$
are positive real numbers. If the $q_i$ lie in a subset $\Gamma$ of $\RR_{>0}$, 
we say that $\frA$ is
a \emph{$\Gamma$-ideal}. 
Given two $\RR$-ideals $\frA$ and $\frB$ on $X$ and a point $x\in {\rm Supp}(\frA)$, 
if $(X,\frB)$ is log canonical, then one defines 
the {\it mixed log canonical
threshold} $\lct_{(X,\frB),x}(\frA)$ to be the largest $c$ such that the pair $(X,\frA^c\frB)$ is 
log canonical at $x$. One reduces to the more familiar setting of log canonical
thresholds when $\frB = \O_X$.

In order to study limits of log canonical thresholds, 
the basic ingredient in the methods used in \cite{dFM,Kol1,dFEM} 
is the construction of generic limit ideals. 
The main obstruction in proving Shokurov's Conjecture in its general form comes from the problem
of constructing a ``generic limit ambient space'' where the generic limit ideal
should live. From this point of view, the advantage in the smooth and 
locally complete intersection cases
is that one can easily reduce to work with one fixed polynomial ring, so that
in the end, in order to 
construct generic limit ideals, it suffices to take a field extension and complete 
the ring at the origin. 

In this paper we consider the case of bounded singularities. 
We say that a collection of germs of algebraic varieties $(X_i,x_i)$ has 
\emph{(formally) bounded singularities} if there are integers 
$m$ and $N$ such that for every $i$ there is
a subscheme $Y_i$ in $\AAA^N$ whose ideal is defined by equations of degree $\leq m$,
and a point $y_i\in Y_i$ such that $\widehat{\cO_{X_i,x_i}}\simeq\widehat{\cO_{Y_i,y_i}}$.
Equivalently, this means that 
there exists a morphism $\pi\colon {\mathcal Y}\to T$, such that for every $i$ there 
is a closed point $t_i\in T$ and a point $y_i$ in the fiber ${\mathcal Y}_{t_i}$ over $t_i$
such that $\widehat{\cO_{X_i,x_i}}\simeq\widehat{\cO_{{\mathcal Y}_{t_i},y_i}}$.
At the moment this appears to be the
most general context where the approach through generic limits can be put to work, 
and it seems likely that new methods will be needed to attack the conjecture in its general form. 

We can now state our main result. 

\begin{theorem}\label{main}
If $\Gamma\subset\RR_{>0}$ is a DCC set, then
there is no infinite strictly increasing sequence of log canonical thresholds
$$\lct_{(X_1,\frB_1),x_1}(\frA_1)<\lct_{(X_2,\frB_2),x_2}(\frA_2)<\ldots,$$
where the $(X_i,x_i)$ form a collection of klt varieties with bounded singularities
and $\frA_i$, $\frB_i$ are $\Gamma$-ideals such that all pairs $(X_i,\frB_i)$ are log canonical.
\end{theorem}

The result in \cite{dFEM} covers the case when the $X_i$ are assumed to be nonsingular
(or more generally, locally complete intersection),
and $\Gamma=\ZZ_{>0}$. In the nonsingular setting, the first result in this direction was obtained 
in \cite{dFM}, where it was shown using ultrafilter constructions that every limit of invariants of the form $\lct(X_i, \fra_i)$, with $\dim(X_i)=n$ for all $i$, is again an invariant of the same form. Koll\'{a}r replaced in \cite{Kol1} the ultrafilter approach
by a generic limit construction, using more traditional algebro-geometric methods. In addition, 
using the results in \cite{BCHM} he proved a semicontinuity property for log canonical thresholds
that allowed him to treat a special case of the conjecture, namely when the log canonical threshold of the limit is computed by a divisor with center at one point. In \cite{dFEM} we gave a more elementary proof of Koll\'{a}r's semicontinuity result, and showed that this can be used in fact to deduce the full statement of the above theorem when all $X_i$ are nonsingular (and $\Gamma=\ZZ_{>0}$).
More general cases, such as when the $X_i$ are locally complete intersection or have
quotient singularities, were deduced from the nonsingular case in a direct fashion.

Regarding the statement of Theorem~\ref{main}, we emphasize that while the category of varieties
with bounded singularities is quite large, it is not large enough for the applications to the Minimal Model Program. More precisely, it is not the case that, for example, terminal $\QQ$-factorial singularities in a fixed dimension have bounded singularities. This is simply because one can construct 
quotient singularities that satisfy these properties, and of arbitrary embedding dimension. Furthermore, as Miles Reid pointed out to us, starting from dimension five there are families of terminal
singularities of arbitrary high embedding dimension that are not nontrivial quotients by finite group actions.

The proof of Theorem~\ref{main} is based on the generic limit construction from
\cite{Kol1}, suitably adapted to our setting.
The main novelty in the proof of the above theorem 
is the simultaneous construction of a generic limit of the ambient spaces and of the ideal
sheaves involved. 
The fact that the embedded dimension of the varieties is bounded is necessary 
in order to have the ``generic limit variety'' being defined
by an ideal in a power series ring with finitely
many variables. We use the fact that the singularities themselves are bounded
to guarantee that such limit variety 
is normal, $\Q$-Gorenstein, and klt. 

There are however several technical difficulties that
arise when working in this general setting. 
Some of these technical points are of a more general nature, 
not necessarily related to the main topic of the paper, and therefore
their treatment will be deferred to the end of the paper. This
will result in two appendices.

The first technical difficulty comes from the fact that we work with 
singular varieties in the formal setting. 
It has became evident since \cite{dFM} that the formal setting is very natural 
when dealing with this kind of problems.
However, while in the previous papers \cite{dFM,Kol1,dFEM}
the formal setting always occurred 
at regular points, in the present paper we need to work with
possibly singular schemes of finite type over complete local rings. 

The generic limit construction that is essential for proving Theorem~\ref{main} 
requires us to develop the 
theory of log canonical pairs in a slightly more general framework than usual: 
working with $\R$-ideals 
on schemes of finite type over a complete Noetherian ring (of characteristic zero). 
This is 
the case since starting with a sequence as in the theorem, 
the generic limit construction provides 
an ambient space that is the spectrum of a complete local ring. While this ring is the
completion 
at a closed point of a scheme of finite type over a field, we need to consider ideals in this 
ring that do not come via completion from the finite type level. 
This will require us to extend the basic results on log 
canonical thresholds to this setting.

In particular, in order to have the notion of relative canonical class in this setting, we 
will need to develop a theory of sheaves of differentials that is adapted to this context.
This part is extracted from the main body of the paper and forms the first appendix. 

The second appendix is devoted to another technical complication 
arising in the proof of Theorem~\ref{main}. The problem comes 
from the fact that we need to be able to bound the Gorenstein index of the varieties 
appearing in the statement in order to conclude that the limit variety is
$\Q$-Gorenstein. 
To this end, we prove a general result on the behavior of the Gorenstein index
in bounded families (see Theorem~\ref{thm:Q-Gor}). 
This result is of independent interest, 
and a slightly simplified version of it can be stated as follows.

\begin{theorem}
Let $f \colon X \to T$ be a morphism of normal complex varieties such that
every fiber of $f$ is a normal variety with rational singularities.
Then there is a nonempty Zariski open subset $T^\o \subseteq T$ and a positive integer $s$
such that for every point $x \in X$ with $t = f(x)\in T^\o$, the fiber $X_t$
is $\Q$-Gorenstein at $x$ if and only if the total space $X$ is $\Q$-Gorenstein at $x$;
furthermore, in this case both $sK_X$ and $sK_{X_t}$ are Cartier at $x$.
\end{theorem}

A variant of the result 
holding over arbitrary algebraically closed fields of characteristic zero
is given in Theorem~\ref{Gor_index}. This will be the version of the result
applied in the proof of Theorem~\ref{main}.

\noindent{\bf Acknowledgment}. We are grateful to Mark de Cataldo for comments 
and suggestions related to the material in Appendix~B, and to J\'{a}nos Koll\'{a}r and Miles
Reid for sharing with us some interesting examples of singularities.

\section{Log canonical pairs on schemes of finite type over a complete local ring}

Throughout this section, let $k$ be a field of characteristic
zero, and let $R = k\llbracket x_1,\ldots,x_n\rrbracket$. 
Our goal is to define and prove the basic properties of
log canonical and log terminal pairs when the ambient space
is a scheme of finite type over $R$. 
Of course, the definitions parallel to the ones in the case of 
schemes of finite type over fields.
The main difference is that in order to define the relative canonical class, 
we need to work with sheaves of {\it special differentials} as defined in Appendix~A. 
The theory of special differentials
enables us to define the notion of relative canonical divisor in this setting
(see in particular Lemma~\ref{lem3_1}). 
Once we have the notion of relative canonical class, 
the theory of singularities of pairs can be built in the same way
as in the case of schemes of finite type over a field, for which we refer to \cite{Kol2}. 
However, we will need to work with $\RR$-ideals (as opposed to $\RR$-divisors), 
hence we give all definitions in this setting. 

In the following, let $X$ be a scheme of finite type over $R$. An \emph{$\RR$-ideal}
on $X$ is a formal product $\frA=\prod_{i=1}^r\fra_i^{p_i}$,
where $r$ is a positive integer, each $\fra_i$ is a nonzero (coherent) ideal sheaf on $X$, 
and the $p_i$ are positive real numbers. 
We call $\frA$ a \emph{proper} $\RR$-ideal if there is $i$ with $\fra_i\neq\cO_X$.
If the $p_i$ are required to lie in some
subset $\Gamma\subseteq\RR_{>0}$, then $\frA$ is called a $\Gamma$-ideal.

The notions we are interested in are invariant with respect to the equivalence relation
that identifies two $\RR$-ideals if they have the same order of vanishing along 
all divisorial valuations.
More precisely, we consider all
proper birational morphisms $\pi\colon Y\to X$, with $Y$ normal, and all 
prime divisors $E$ on $Y$. Every such $E$ defines a valuation 
$\ord_E$ of the function field of $X$. The image of $E$ on $X$ is the \emph{center}
of $E$ on $X$, and it is denoted by $c_X(E)$.
If $\fra$ is an ideal sheaf on $X$, then 
$\ord_E(\fra)$ is the minimum of $\ord_E(w)$, where $w$ varies over the sections of 
$\fra$ defined at the generic point of $c_X(E)$. If $\frA=\prod_{i=1}^r\fra_i^{p_i}$ is
an $\RR$-ideal on $X$, then
$$\ord_E(\frA):=\sum_{i=1}^r p_i\cdot\ord_E(\fra_i).$$
The equivalence relation identifies $\frA$ and $\frA'$ 
whenever $\ord_E(\frA)=\ord_E(\frA')$ for every $E$ as above. 

\begin{remark}
By Theorem~\ref{thm1_temkin} below, whenever we consider a valuation $\ord_E$ as above,
we may assume that the model $Y$ on which $E$ lies is nonsingular.
\end{remark}

\begin{example}\label{example1}
If $\frA=\prod_{i=1}^r\fra_i^{p_i}$, where all $p_i\in\QQ$, then $\frA$ is identified with
$\frb^{1/m}$, where $m$ is a positive integer such that $mp_i\in\ZZ$ for all $i$, and
$\frb=\prod_{i=1}^r\fra_i^{mp_i}$. Furthermore, two such $\QQ$-ideals $\frb_1^{1/m}$ and $\frb_2^{1/m}$ 
are identified if and only if for some positive integer $q$, the ideals $\frb_1^q$ and 
$\frb_2^q$ have the same integral closure.
\end{example}

The product of $\RR$-ideals is defined in the obvious way, by concatenating the factors.
Similarly, if $\frA=\prod_{i=1}^r\fra_i^{p_i}$ is as above, and $q\in\RR_{>0}$, 
then $\frA^q:=\prod_{i=1}^r\fra_i^{qp_i}$. Note that these operations preserve the above equivalence classes.

Suppose now that $X$ is normal, and let $\frA=\prod_{i=1}^r\fra_i^{p_i}$ 
be an $\RR$-ideal on $X$. 
A \emph{log resolution} of $(X,\frA)$ is a
log resolution for the pair $(X,\prod_{i=1}^r\fra_i)$. 
Recall that this is a proper birational morphism $\pi\colon Y\to X$, with
$Y$ nonsingular, such that the exceptional locus of $\pi$ and the inverse images of the subschemes $V(\fra_i)$
are Cartier divisors, and all these divisors have simple normal crossings. 
Since we are in characteristic zero, the existence of log resolutions in our setting is guaranteed by the results in
\cite{Temkin}. For completeness, we explain how to get log resolutions from the
results in \emph{loc. cit.} The two theorems below hold for arbitrary 
quasi-excellent schemes\footnote{A scheme is quasi-excellent if it is covered by affine open subsets of the form 
$\Spec(A_i)$, with each $A_i$ a quasi-excellent ring; the definition of quasi-excellent ring is
similar to that of excellent ring, but one does not require the ring to be universally catenary,
see \cite[p. 260]{Matsumura}.},
so they hold in particular in our setting, for schemes of finite type over $R$. 

\begin{theorem}\label{thm1_temkin}(\cite{Temkin})
For every integral scheme $X$ of finite typer over $R$, there is a proper birational morphism
$\pi\colon Y\to X$ with $Y$ nonsingular. Furthermore, we may construct $\pi$ such that
it is an isomorphism over $X_{\rm reg}$.
\end{theorem}

\begin{theorem}\label{thm2_temkin}(\cite{Temkin})
If $X$ is a nonsingular scheme as above, and $D$ is an effective divisor on $X$,
then there is a proper birational morphism $\pi\colon Y\to X$ such that $Y$ is nonsingular
and $\pi^*(D)$ has simple normal crossings. Furthermore, we may assume that $\pi$
is an isomorphism over $X\smallsetminus {\rm Supp}(D)$.
\end{theorem}

Let us explain how to combine these two theorems in order to get log resolutions. 
Suppose that $(X,\frA)$ is a pair as above, with $\frA=\prod_{i=1}^r\fra_i^{p_i}$,
and $X$ normal. 
We first apply Theorem~\ref{thm1_temkin} to construct $\pi_1\colon Y_1\to X$
proper and birational, and such that $Y_1$ is nonsingular. Since $X$ is normal,
there is an open subset $U\subseteq X$ such that $\pi_1$ is an isomorphism over $U$,
and $Z:=X\smallsetminus U$ has ${\rm codim}(Z,X)\geq 2$. We note that if $\phi\colon Y\to
Y_1$ is proper and birational, with $Y$ nonsingular, and if $\phi^{-1}(\pi_1^{-1}(Z))$ is a 
divisor, then 
\begin{equation}\label{eq_exceptional_locus}
{\rm Exc}(\pi_1\circ\phi)={\rm Exc}(\phi)\cup\phi^{-1}(\pi_1^{-1}(Z)).
\end{equation}
In particular, this exceptional locus is a divisor (recall that ${\rm Exc}(\phi)$ is a 
divisor; this follows for instance from Lemma~\ref{lem3_1}). 

We blow-up successively along $\prod_{i=1}^r\fra_i$, and along the inverse image of $Z$,
to get $\pi_2\colon Y_2\to Y_1$. We now apply one more time Theorem~\ref{thm1_temkin}
to get a proper and birational morphism $\pi_3\colon Y_3\to Y_2$ with $Y_3$ nonsingular. 
Furthermore, we do this so that $\pi_3$ is an isomorphism over $(Y_2)_{\rm reg}$.
It follows from (\ref{eq_exceptional_locus}) that $E:={\rm Exc}(\pi_1\circ\pi_2\circ\pi_3)$
is an effective divisor on $Y_3$, and we have effective divisors $E_i$ on $Y_3$
such that $\fra_i\cdot\cO_{Y_3}=\cO_{Y_3}(-E_i)$. 
Furthermore, if $Z'=(\pi_1\circ\pi_2\circ\pi_3)^{-1}(Z)$, then by construction
\begin{equation}\label{eq_exceptional_locus2}
{\rm Supp}(Z')\subseteq 
{\rm Supp}(E)\subseteq {\rm Supp}(Z')\cup {\rm Supp}(E_1+\cdots+E_r).
\end{equation}
We apply Theorem~\ref{thm2_temkin}
to get a proper birational morphism $\pi_4\colon Y\to Y_3$ with $Y$ nonsingular,
and such that $\pi_4^*(Z'+E_1+\cdots+E_r)$ has simple normal crossings. Furthermore,
we may and will assume that this is an isomorphism over $Y_3\smallsetminus 
{\rm Supp}(Z'+E_1+\cdots+E_r)$. We let $\pi\colon Y\to X$ be the composition. 
Using (\ref{eq_exceptional_locus}) and (\ref{eq_exceptional_locus2}) we see that ${\rm Exc}(\pi)$ is a divisor,
and that it is contained in the support of $\pi_4^*(Z'+E_1+\cdots+E_r)$. Therefore $\pi$
is a log resolution of $(X,\frA)$. 
A similar argument can be used to show that any two log resolutions
of $(X,\frA)$ are dominated by a third one.

\bigskip

Suppose now that $X$ is $\QQ$-Gorenstein, and let
$\pi\colon Y\to X$ be a log resolution of a pair $(X,\frA)$,
where $\frA=\prod_{i=1}^r\fra_i^{p_i}$.
Let $K_{Y/X}$ be the relative canonical divisor as defined in Appendix~A
(cf. Lemma~\ref{lem3_1}).
Since $K_{Y/X}$ is supported on the exceptional locus, it follows that
there is  a simple normal crossings
divisor $\sum_{j=1}^{\ell}E_j$ on $Y$ with
\begin{equation}\label{numerical_invariants}
K_{Y/X}=\sum_{j=1}^{\ell}\kappa_jE_j,\quad\fra_i\cdot\cO_Y=
\cO_Y\Big(-\sum_{j=1}^{\ell}\alpha_{i,j}E_j\Big)\,\text{for}\,1\leq i\leq r.
\end{equation}
The pair $(X,\frA)$ is called \emph{log canonical} if 
\begin{equation}\label{ineq_invariants}
\kappa_j+1\geq\sum_{i=1}^r\alpha_{i,j}p_i=\ord_{E_j}(\frA)
\end{equation}
for all $j$. If all inequalities in (\ref{ineq_invariants}) are strict, the pair is 
\emph{Kawamata log terminal} (or
\emph{klt}, for short). If $\frA=\cO_X$, we simply say that $X$ is log canonical or klt, 
respectively.
Note that the definitions are independent of the representative for $\frA$
is our equivalence class.
The fact that the definition is independent of the log resolution follows in the same
way as in the
case of schemes of finite typer over a field. The key ingredients are given by Lemma~\ref{lem3_1} iii), and the fact that
any two log resolutions can be dominated by a third one.

\begin{remark}\label{rem3_2}
It follows from Remark~\ref{rem3_1} that if $X$ is nonsingular, then the log canonicity of 
a pair $(X,\fra)$ is independent of the $R$-scheme structure of $X$. Again, it is not clear to us
whether the same remains true if $X$ is singular (however, see Remark~\ref{rem_invariance}
below for one case when this holds, which is the one that concerns us most).
\end{remark}

Let $\frA=\prod_{i=1}^r\fra_i^{p_i}$ 
and $\frB=\prod_{i=1}^s\frb_i^{q_i}$ be $\RR$-ideals on $X$, with $\frA$ a proper ideal.
If the pair $(X,\frB)$ is log canonical, then we define the 
\emph{mixed log canonical threshold} $\lct_{(X,\frB)}(\frA)$ (written also as $\lct_{\frB}(\frA)$ when there
is no ambiguity about the ambient scheme) as the largest $c\geq 0$ such that
$(X,\frA^c\frB)$ is log canonical. If $\frB=\cO_X$, then we simply write $\lct(X,\frA)$ or 
$\lct(\frA)$, and we call it the \emph{log canonical threshold} of $\frA$.
If $\pi$ as above is a log resolution of $(X, \frA\cdot\frB)$, and if 
$\frb_i\cdot\cO_Y=
\cO_Y(-\sum_j\beta_{i,j}E_j)$ for $1\leq i\leq s$, then 
\begin{equation}\label{formula_lct}
\lct_{(X,\frB)}(\frA)=\min_j\frac{\kappa_j+1-\sum_{i=1}^s\beta_{i,j}q_i}{\sum_{i=1}^r\alpha_{i,j}p_i}.
\end{equation}
Note that since $\frA$ is assumed to be proper, there are
$i$ and $j$ such that $\alpha_{i,j}>0$, hence the above minimum is finite.
If $E_j$ is such that the minimum in (\ref{formula_lct}) is achieved, we say that $E$ \emph{computes} 
$\lct_{(X,\frB)}(\frA)$.

We also consider a local version of the above invariant. 
If $\frA=\prod_{i=1}^r\fra_i^{p_i}$ is an $\RR$-ideal on $X$, we denote by 
${\rm Supp}(\frA)$ the union of the closed subsets of $X$ defined by the ideals $\fra_i$. 
If $x\in {\rm Supp}(\frA)$, and $(X,\frB)$ is log canonical in some open neighborhood of $x$, 
then $\lct_{(X,\frB),x}(\frA)$ is the largest $c\geq 0$ such that $(X,\frA^c\frB)$ is log canonical in some neighborhood of $x$. If $\frB=\cO_X$, we write $\lct_x(X,\frA)$ or $\lct_x(\frA)$. 
Of course, $\lct_{(X,\frB),x}(\frA)$ can be described by a formula analogous to (\ref{formula_lct}),
in which the minimum is over those $j$ such that $x\in c_X(E_j)$.

For simplicity, we will state most of the basic properties of log canonical thresholds only in the
unmixed setting, since we will only need these versions.
The following lemma is a simple consequence of the formula
for the log canonical threshold in terms of a log resolution.

\begin{lemma}\label{simple_lemma}
Suppose that $X$ is log canonical, and let
$\frA=\prod_{i=1}^r\fra_i^{p_i}$ and $\frB=\prod_{i=1}^s\frb_i^{q_i}$
be proper $\RR$-ideals on $X$. If $s\leq r$, and
$\fra_i\subseteq\frb_i$ and $p_i\geq q_i$ for all $i\leq s$, then
$\lct(\frA)\leq\lct(\frB)$. A similar assertion holds for the local
version of log canonical thresholds.
\end{lemma}

It is sometimes convenient to reduce the study of log canonical thresholds of $\RR$-ideals
to that of $\QQ$-ideals (hence to that of usual ideals). This can be done using the following two lemmas (the first one deals with the log canonical threshold, while the second one treats the divisors computing the log canonical threshold).

\begin{lemma}\label{lem_reduction}
Assume that $X$ is log canonical, and let
$\frA=\prod_{i=1}^r\fra_i^{p_i}$ be a proper $\RR$-ideal on $X$.
If $(p_{i,m})_{m\geq 1}$ are sequences of positive real numbers with $\lim_{m\to\infty}
p_{i,m}=p_i$ for every $i\leq r$, and if $\frA_m=\prod_{i=1}^r\fra_i^{p_{i,m}}$,
then $\lim_{m\to\infty}\lct(\frA_m)=\lct(\frA)$.
 A similar assertion
 holds for the local version of log canonical threshold $\lct_x(\frA)$.
\end{lemma}

\begin{proof}
The assertion follows immediately from formula (\ref{formula_lct}).
\end{proof}

\begin{lemma}\label{lem_reduction2}
Suppose that $X$ is log canonical, and let
$E$ be a divisor computing $\lct_{x}(\frA)=\lambda$, for some $\RR$-ideal 
$\frA=\prod_{i=1}^r\fra_i^{p_i}$ on $X$, containing $x$ in its support.
Then one can
find sequences of rational numbers $(p_{i,m})_{m\geq 1}$
with $\lim_{m\to\infty}p_{i,m}=\lambda p_i$, and such that 
if we put $\frA_m=\prod_{i=1}^r\fra_i^{p_{i,m}}$, then
$\lct_{x}(\frA_m)=1$ and
$E$ computes 
$\lct_{x}(\frA_m)$ for every $m$.
\end{lemma}

\begin{proof}
Let $\pi\colon Y\to X$ be a log resolution of $\frA$ such that $E$ is a divisor on $Y$. 
With the notation
in (\ref{numerical_invariants}), after restricting to a suitable open neighborhood of $x$,
we may assume that $x\in c_X(E_j)$ for all $j$.
Consider the rational polyhedron
$$P=\big\{(u_1,\ldots,u_r)\mid\kappa_j+1\geq\sum_{i=1}^r\alpha_{i,j}u_i\text{ for all }j\big\}.$$ 
 If $E=E_{j_0}$, then we see that $(\lambda p_1,\ldots,\lambda p_r)$ lies on the face $P_E$ 
of $P$ defined by $\kappa_{j_0}+1=\sum_i\alpha_{i,j_0}u_i$.
Since $P_E$ is itself a rational polyhedron, it follows that $(\lambda p_1,\ldots,\lambda p_r)$ can be written as the limit of a sequence $(p_{1,m},\ldots,p_{r,m})\in P_E\cap\QQ^r$. It is clear that this
sequence satisfies our requirements.
\end{proof}

The following lemma allows one to reduce the study of the log canonical threshold to the
case when this invariant is computed by a divisor with center equal to a closed point.
The proof is the same as that of \cite[Lemma~5.2]{dFEM}, so we omit it.

\begin{lemma}\label{lem_red_zero_dimensional}
Suppose that $X$ is log canonical, $\frA$ is a proper $\RR$-ideal on $X$, and
$x\in X$ is a closed point defined by the ideal $\frm_x$.
If $c=\lct_x(\frA)$, then there is a nonnegative real number $t$
such that $c=\lct_x(\frm_x^t\cdot\frA)$, and this log canonical threshold is computed by
a divisor $E$ over $X$ having center equal to $x$.
\end{lemma}

We will mainly be interested in the case when the ambient variety is either a scheme of finite type over a field, or 
the spectrum of the completion of the local ring of such a scheme at a closed point. The following proposition gives the compatibility of the log canonical threshold with respect to 
taking such a completion. Suppose that $X$ is a scheme of finite type over $k$ and
$x\in X$ is a closed point, 
and consider $g\colon Z=\Spec(\widehat{\cO_{X,x}})\to X$. 
If $\frA=\prod_{i=1}^r\fra_i^{p_i}$ is an 
$\RR$-ideal on $X$, we denote by $\widehat{\frA}$ the $\RR$-ideal  
$\prod_{i=1}^r(\fra_i\cO_Z)^{p_i}$.
We consider the Cartesian diagram
\begin{equation}\label{diag2}
\xymatrix{
W \ar[r]^h \ar[d]_f & Y \ar[d]^\pi \\
Z={\rm Spec}(\widehat{\cO_{X,x}}) \ar[r]^(.7)g & X
}
\end{equation}
where $\pi\colon Y\to X$ is a proper birational morphism with $Y$ nonsingular.

\begin{remark}\label{rmk:diagram}
Since $g$ is a regular morphism (see \cite[Chapter 32]{Matsumura}), it follows that 
so is $h$. Recall that a morphism of Noetherian schemes is regular if it is flat, and has
geometrically regular fibers. Since $g$ is regular, we see that $X$ is normal around $x$
if and only if $Z$ is normal, and since $h$ is regular, we see that $W$ is nonsingular. 
\end{remark}

\begin{proposition}\label{prop3_2}
With the above notation, the following hold:
\begin{enumerate}
\item[i)] The pair
$(X,\frB)$ is log canonical ${\rm (}$respectively, klt${\rm )}$ in a neighborhood of $x$ if 
and only if $(Z,\widehat{\frB})$ is log canonical ${\rm (}$respectively, klt${\rm )}$. 
\item[ii)] If $(X,\frB)$ is log canonical in a neighborhood of $x\in {\rm Supp}(\frA)$, 
then $\widehat{\frA}$
is a proper $\RR$-ideal and 
$$\lct_{(X,\frB),x}(\frA)=\lct_{(Z,\widehat{\frB})}(\widehat{\frA}).$$
\item[iii)] Under the assumptions in ${\rm ii)}$, if $E_i$ and $F$ are as in Remark~\ref{divisor_over_point}, then 
$F$ computes $\lct_{(X,\frB),x}(\frA)$ if and only if $E_i$ computes 
$\lct_{(Z,\widehat{\frB})}(\widehat{\frA})$.
\end{enumerate}
\end{proposition}

\begin{proof}
Note that if $\pi$ is a log resolution
of $(X,\frA\cdot\frB)$, then $f$ is a log resolution of $(Z,\widehat{\frA}\cdot\widehat{\frB})$
because $h$ is a regular morphism. Furthermore, if $F$
is a nonsingular prime divisor on $Y$ such that $x\in c_X(F)$, 
and $E$ is a component of $h^*(F)$, then $E$ is a nonsingular prime divisor on $W$.
It is clear that
the coefficient of $F$ in a simple normal crossings divisor $D$ on $Y$ is equal to the coefficient of 
$E$ in $h^*(D)$. We now deduce the assertions in the proposition from Proposition~\ref{prop3_1}.
\end{proof}

\begin{remark}\label{rem_invariance}
In the setting of the proposition, it follows from Proposition~\ref{prop3_1} that the divisor
$K_{W/Z}$ does not depend on the presentation of $\widehat{\cO_{X,x}}$
via Cohen's Structure Theorem. Furthermore, if $\frA'$ is an arbitrary $\RR$-ideal
on $Z$ (not necessarily coming from $X$), and if we consider a log resolution $W'\to Z$
of $(Z,\frA')$, then it follows from Lemma~\ref{lem3_2} and Remark~\ref{rem3_1}
that $K_{W'/Z}$ is independent
of the presentation of $\widehat{\cO_{X,x}}$. Therefore, the (mixed) log canonical thresholds
on $Z$ are independent of this presentation.
\end{remark}

\begin{remark}
Note that in the setting of the proposition, if $X$ is nonsingular, then the conclusion of the
proposition also holds if the localization is  at a non-closed point. Indeed, when we deal with nonsingular schemes, then we do not need to consider $\cO(K_X)$, 
as the divisors $K_{Y/X}$ and $K_{W/Z}$
can be computed using the $0^{\rm th}$ Fitting ideals of the corresponding sheaves of relative differentials, and $h^*\Omega_{Y/X}
\simeq\Omega_{W/Z}$.
\end{remark}

The following lemma concerns the behavior of singularities of pairs for 
schemes of finite type over a field under the extension of the ground field. In particular,
it allows us to reduce the study of singularities of such  pairs to the case when the ground field is algebraically closed.

\begin{lemma}\label{alg_closed}
Let $X$ be a normal scheme of finite type over a field $k$. If $K$ is a
field extension of $k$, and if $\phi\colon \overline{X}=X\times_{\Spec(k)}{\Spec(K)}
\to X$ is the projection,
then for every $\RR$-ideal $\frA=\prod_{i=1}^r\fra_i^{p_i}$ on $X$, and every
$\overline{x}\in \overline{X}$ and $x=\phi(\overline{x})$, we have:
\begin{enumerate}
\item[i)] $rK_X$ is Cartier at $x$ if and only if $rK_{\overline{X}}$ is Cartier at 
$\overline{x}$.
\item[i)] The pair $(X,\frA)$ is log canonical ${\rm (}$respectively, klt${\rm )}$
 in some neighborhood of $x$ if and only if $(\overline{X},\overline{\frA})$
is log canonical ${\rm (}$respectively, klt${\rm )}$ 
in some neighborhood of $\overline{x}$, where 
$$\overline{\frA}=\prod_{i=1}^r(\fra_i\cO_{\overline{X}})^{p_i}.$$
\item[iii)] If $X$ is log canonical at $x$, then $\lct_x(X,\frA)=\lct_{\overline{x}}(\overline{X},\overline{\frA})$.
\item[iv)] If $F$ is a divisor that computes $\lct_x(X,\frA)$, and if $E$ is a component of 
the divisor $\overline{F}=F\times_{\Spec(k)}\Spec(K)$ on $\overline{X}$ whose center contains 
$\overline{x}$, then
$E$ computes $\lct_{\overline{x}}(\overline{X},\overline{\frA})$.
\end{enumerate}
\end{lemma}

\begin{proof}
Note that the fibers of $\phi$ are disjoint unions of zero-dimensional, reduced schemes.
Since we are in characteristic zero and $\phi$ is flat, we deduce from this fact that
 $\phi$ is regular. 
In particular, $\overline{X}$ is normal (though it might be disconnected), 
and $\phi^{-1}(X_{\rm reg})=\overline{X}_{\reg}$. It is also easy to deduce that
if $\cF$ is a reflexive sheaf on $X$, then $\phi^*(\cF)$ is reflexive, and 
$\cF$ is generated by one element around $x$ if and only if $\phi^*(\cF)$ is generated by one
element around $\overline{x}$. Since 
$\Omega_{\overline{X}/\overline{k}}\simeq\phi^*(\Omega_{X/k})$, this implies that
we can take 
$\phi^*(K_X)=K_{\overline{X}}$, and $rK_X$ is Cartier at $x$ if and only if
$rK_{\overline{X}}$ is Cartier at $\overline{x}$.

Suppose now that $X$ is $\QQ$-Gorenstein, and
let $f\colon Y\to X$ be a log resolution of $(X,\frA)$. If $\overline{Y}=Y\times_{\Spec(k)}\Spec(K)$, then we have a Cartesian diagram
$$
\xymatrix{
\overline{Y} \ar[d]_g \ar[r]^\psi & Y \ar[d]^f \\
\overline{X} \ar[r]^\phi & X .
}
$$
If $f$ is an isomorphism over $U\subseteq X$, then $g$ is an isomorphism over
the dense open subset $\phi^{-1}(U)$ of $\overline{X}$.
Since $\psi$ is regular, arguing as in the proofs of Propositions~\ref{prop3_1}
and \ref{prop3_2}, we see that $g$ is a log resolution of $(\overline{X},\overline{\frA})$,
and we have $K_{\overline{Y}/\overline{X}}=\psi^*(K_{Y/X})$. 
Note also that if $E$ is a prime nonsingular divisor on $Y$ such that
$x\in c_X(E)$, then there is a component $F$ of $\overline{E}=\psi^*(E)$ such that
$\overline{x}\in c_{\overline{X}}(F)$. For every such $F$, the valuation 
$\ord_F$ restricts to the valuation $\ord_E$ on the function field of $X$.
The remaining assertions in the proposition
are easy consequences of these observations.
\end{proof}

\bigskip

We now give some further properties of log canonical thresholds that will be used in the proof of our main result. These generalize corresponding results for schemes of finite type over a field, and for usual ideals. Let us fix the notation. In what follows $X$ is a log canonical scheme of finite type over a field $k$. Let $x\in X$ be a closed point with $\dim(\cO_{X,x})=n$, and let $g\colon Z={\rm Spec}(\widehat{\cO_{X,x}})\to X$
be the canonical morphism. We have seen in Proposition~\ref{prop3_2} that if $\frA$
is an $\RR$-ideal on $X$ such that $x\in {\rm Supp}(\frA)$, then
$\lct(Z,\widehat{\frA})=\lct_x(X,\frA)$. The following lemma allows us to approximate every log canonical threshold on $Z$ by log canonical thresholds of pull-backs of $\RR$-ideals on $X$.

\begin{proposition}\label{prop_limit}
Let $\frB=\prod_{i=1}^r\frb_i^{q_i}$ be an $\RR$-ideal on $Z$. If $\frm$ is the ideal
defining the closed point on $Z$, and if we put 
$\frB_d=\prod_{i=1}^r(\frb_i+\frm^d)^{q_i}$, then 
$$\lim_{d\to\infty}\lct(Z,\frB_d)=\lct(Z,\frB).$$
Furthermore, if there is a divisor $E$ over $Z$ computing $\lct(Z, \frB)$ and with center 
equal to the closed point, then $\lct(Z, \frB_d)=\lct(Z, \frB)$ for $d\gg 0$.
\end{proposition}

\begin{proof}
The proof follows verbatim the proof of \cite[Proposition~2.5]{dFM} (the hypothesis in 
\emph{loc. cit.} that the ambient scheme is nonsingular does not play any role). The key step
is to show that 
$$\lct(Z, \frB)=\inf_F\frac{\ord_F(K_{W/Z})+1}{\ord_F(\frB)},$$
where the infimum is over the divisors $F$ over $Z$ (lying on some $W$) having center equal to the closed point. We refer to \emph{loc. cit.} for details.
\end{proof}

\begin{remark}\label{rem_prop_limit}
With the notation in the proposition, if $\frn$ denotes the maximal ideal in $\cO_{X,x}$, then
$\frm=\frn\cdot\widehat{\cO_{X,x}}$, and $\cO_{X,x}/\frn^d\simeq\widehat{\cO_{X,x}}/\frm^d$
for every $d$. It follows that, after possibly replacing $X$ by an affine open 
neighborhood of $x$, every $\RR$-ideal $\frB_d$ in the proposition can be written as
$\widehat{\frA_d}$ for some $\RR$-ideal $\frA_d$ on $X$.
\end{remark}

\begin{lemma}\label{lem_kawakita}
Suppose that $X$ is klt.
If $\frn$ is the ideal defining $x\in X$, then $\lct(X,\frn)\leq n$.
\end{lemma}

\begin{proof}
We apply
Lemma~\ref{alg_closed}, with $K=\overline{k}$, the algebraic closure of $k$.
We see that $\lct(X,\frn)=\lct_{\overline{x}}(\overline{X},\overline{\frn})$
for any point $\overline{x}\in\phi^{-1}(x)$.
Since in some neighborhood of $\overline{x}$,  $\overline{\frn}$ is equal to the ideal defining 
$\overline{x}$, we see that we may replace $X$ by $\overline{X}$. Therefore
we may assume that $k$ is algebraically closed. 

As pointed out by Kawakita, the bound now follows from the proof of
\cite[Theorem~2.2]{Kawakita}. For the sake of the reader, we briefly recall the argument.
After replacing $X$ by its index one cover corresponding to $\cO(K_X)$, we may assume that
$K_X$ is Cartier. 
 Let $f\colon Y\to X$
be a log resolution of $(X,\frn)$, and write $\frn\cdot\cO_Y=\cO_Y(-E)$, with $E=\sum_im_iE_i$.
We can choose $F=E_{i_0}$ such that $\cO(-E)\vert_F$ is big and nef. For every $m\geq 1$, 
we have an exact sequence
$$0\to\cO(K_{Y/X}-mE)\to\cO(K_{Y/X}-mE+F)\to\cO(K_F)\otimes\cO(-mE)\vert_F\to 0.$$
By the Kawamata-Viehweg Vanishing Theorem, it follows that 
$$P(m):=h^0(F,\cO(K_F)\otimes\cO(-mE)\vert_F)=\chi(F,\cO(K_F)\otimes\cO(-mE)\vert_F),$$
and this is a polynomial of degree $(n-1)$ since $((-E)^{n-1}\cdot F)>0$. It follows that
there is an integer $s$, with $1\leq s\leq n$ such that $P(s)\neq 0$. 

On the other hand,
another application of the Kawamata-Viehweg Vanishing Theorem implies that
$R^1f_*\cO(K_{Y/X}-sE)=0$, hence
 the sequence 
$$0\to f_*\cO(K_{Y/X}-sE)\to f_*\cO(K_{Y/X}-sE+F)\to f_*(\cO(K_F)\otimes\cO(-sE)\vert_F)
\to 0$$
is exact.
It follows that $f_*\cO(K_{Y/X}-sE)\neq f_*\cO(K_{Y/X}-sE+F)$, hence there is 
a rational function $\phi$ on $X$ such that $D:={\rm div}_Y(\phi)+K_{Y/X}-sE+F\geq 0$, and the coefficient of $F$ in $D$ is zero. Since $K_{Y/X}-sE+F$ is $f$-exceptional, it follows that 
$\phi\in\cO(X)$. Therefore $\ord_F(\phi)\geq 0$, and we conclude that 
$\ord_F(K_{Y/X})+1\leq s\cdot\ord_F(\frn)$, and hence $\lct(X,\frn)\leq s\leq n$.
\end{proof}

\begin{lemma}\label{lem_sum}
If $\fra$ and $\frb$ are ideals on $X$ such that their supports contain $x\in X$, then
$\lct_x(\fra+\frb)\leq\lct_x(\fra)+\lct_x(\frb)$. 
\end{lemma}

\begin{proof}
Arguing as in the proof of the previous lemma, we may assume that $k$ is algebraically closed.
It is now convenient to use the language of multiplier ideals, for which we refer to 
\cite[Chapter~9]{positivity}. The version of the Summation
Theorem from \cite[Corollary~2]{JM} implies that for every $\lambda\geq 0$ we have the following description
for the multiplier ideals of exponent $\lambda$ of a sum of ideals:
\begin{equation}\label{summation}
{\mathcal J}(X,(\fra+\frb)^{\lambda})=\sum_{\alpha+\beta=\lambda}{\mathcal J}(X, \fra^{\alpha}
\frb^{\beta}).
\end{equation}
Recall that $\lct_x(\fra)$ is the smallest $\alpha$ such that $x$ lies in the support of ${\mathcal J}(X, \fra^{\alpha})$. Let $c_1=\lct_x(\fra)$ and $c_2=\lct_x(\frb)$. It is enough to show that
$x$ lies in the support of ${\mathcal J}(X, (\fra+\frb)^{c_1+c_2})$. This follows from (\ref{summation}),
since given $\alpha$, $\beta$ such that $\alpha+\beta=c_1+c_2$, then either
$\alpha\geq c_1$, or $\beta\geq c_2$. In the first case we have
$${\mathcal J}(X, \fra^{\alpha}\frb^{\beta})\subseteq {\mathcal J}(X, \fra^{\alpha})
\subseteq {\mathcal J}(X, \fra^{c_1}),$$
whose support contains $x$. The case $\beta\geq c_2$ is similar. 
\end{proof}

\begin{proposition}\label{prop3_9}
Suppose that $X$ is klt in a neighborhood of $x$.
Let $\frB=\prod_{i=1}^r\frb_i^{q_i}$ be a 
proper $\RR$-ideal on $Z$, and let $\frm$ be the ideal
defining the closed point of $Z$.
\begin{enumerate}
\item[i)] If $\frb_i\subseteq\frm^{s_i}$ for every $i$, then
$\lct(Z,\frB)\leq \frac{n}{s_1q_1+\cdots+s_rq_r}$.
\item[ii)] Suppose that $\frA=\prod_{i=1}^r\fra_i^{q_i}$ is another $\RR$-ideal on $Z$,
and let $\epsilon>0$ be a real number. If $d$ is a positive integer such that 
$d\geq \frac{n}{\epsilon q_i}$ and $\fra_i+\frm^d=\frb_i+\frm^d$ for all $i$, then
$$|\lct(\frA)-\lct(\frB)|\leq\epsilon.$$
\end{enumerate}
\end{proposition}

\begin{proof}
For i), it is clear that if $s=s_1q_1+\cdots+s_rq_r$, then
$$\lct(Z,\frB)\leq \lct(Z,\frm^{s})=
\frac{\lct(Z,\frm)}{s}.$$
By Proposition~\ref{prop3_2} we have $\lct(Z,\frm)=\lct(X,\frn)$, where
$\frn$ is the ideal defining $x\in X$, and we conclude by Lemma~\ref{lem_kawakita}.

For ii), we first show that if $\fra$ and $\frb$ are proper ideals on $Z$ such that 
$\fra+\frm^{\ell}=\frb+\frm^{\ell}$, then $|\lct(\fra)-\lct(\frb)|\leq n/\ell$. Note that by 
Proposition~\ref{prop_limit}, it is enough to show that
$|\lct(\fra+\frm^N)-\lct(\frb+\frm^N)|\leq n/\ell$
for all $N\geq \ell$. Therefore we may assume that there are ideals $\fra'$ and $\frb'$
on $X$ such that $\fra=\widehat{\fra'}$ and $\frb=\widehat{\frb'}$. By Proposition~\ref{prop3_2},
it is enough to show that if $\fra'+\frn^{\ell}=\frb'+\frn^{\ell}$, then
$|\lct_x(\fra')-\lct_x(\frb')|\leq n/\ell$. We deduce from Lemmas~\ref{lem_kawakita}
and \ref{lem_sum} that
$$\lct_x(\fra')\leq\lct_x(\fra'+\frn^{\ell})=\lct_x(\frb'+\frn^{\ell})\leq\lct_x(\frb')+\frac{\lct(\frn)}{\ell}
\leq\lct_x(\frb')+\frac{n}{\ell}.$$
The inequality $\lct_x(\frb')\leq\lct_x(\fra')+\frac{n}{\ell}$ follows by symmetry.

We now prove ii). 
After writing each $q_i$ as a decreasing limit of rational numbers, 
we see using Lemma~\ref{lem_reduction}  that it is enough to prove ii)
when all $q_i\in\QQ$. Let us choose a positive integer $p$ such that all $pq_i$
are integers.

It is enough to show that 
$\lct(\frB)\geq\lct(\frA)-\epsilon$, as the other inequality will follow by symmetry. After replacing
each $\fra_i$ by $\fra_i+\frm^d$, we may assume that $\fra_i=\frb_i+\frm^d$. 
Therefore $\fra_i^{pq_i}=(\frb_i+\frm^d)^{pq_i}$, and this ideal has the same integral closure as 
$\frb_i^{pq_i}+\frm^{dpq_i}$. Since $\prod_i(\frb_i^{pq_i}+\frm^{dpq_i})$
and $\prod_i\frb_i^{pq_i}$ have the same image mod $\frm^{\ell}$, where 
$\ell=dp\cdot\min_iq_i$, we deduce
$$\lct(\frA)=p\cdot\lct\big({\textstyle\prod}_i\fra_i^{pq_i}\big)
={p}\cdot \lct\big({\textstyle\prod}_i(\frb_i^{pq_i}+\frm^{dpq_i})\big)$$
$$\leq {p}\cdot \big(\lct\big({\textstyle\prod}_i\frb_i^{pq_i}\big)+\frac{n}{\ell}\big)
\leq\lct(\frB)+\epsilon,$$
where the last inequality follows from the assumption that $\ell\geq np/\epsilon$.
\end{proof}

The following result is a key ingredient in the proof of our main result. In the case
of schemes of finite type over a field and usual ideals, it was proved in 
\cite{Kol1} and \cite{dFEM}. 

\begin{proposition}\label{semicont}
Consider the proper
$\RR$-ideals $\frA=\prod_{i=1}^r\fra_i^{q_i}$ and 
$\frB=\prod_{i=1}^r\frb_i^{q_i}$ on $Z$, and suppose that $E$ is a divisor that
computes $\lct(\frA)$, having center equal to the closed point of $Z$. 
If $\fra_i+\frp_i=\frb_i+\frp_i$ for all $i$, where $\frp_i=\{u\in \widehat{\cO_{X,x}}\mid \ord_E(u)>\ord_E(\fra_i)\}$, 
then 
$\lct(\frA)=\lct(\frB)$.
\end{proposition}

We will need the following lemma.

\begin{lemma}\label{divisor}
If $E$ is a divisor over $Z=\Spec(\widehat{\cO_{X,x}})$ with center equal to the closed point of $Z$,
then the restriction of $\ord_E$ to the function field of $X$ is of the form $\ord_F$ for some divisor 
$F$ over $X$ with center $x$. Moreover,
one can find a Cartesian diagram as in \eqref{diag2}
such that $F$ appears as a prime nonsingular divisor on $Y$ and $E$ appears as $h^*(F)$.
\end{lemma}

\begin{proof}
We first note that the valuation ring $\cO_v$ of $v=\ord_E$ is essentially of finite type over
$\widehat{\cO_{X,x}}$, and its residue field $k(E)$ has transcendence degree $(n-1)$
over $k$ (recall that $n=\dim(\cO_{X,x})=\dim(\widehat{\cO_{X,x}})$). Indeed, let us realize $E$ as a prime divisor on some nonsingular $T$, with  $\phi\colon T\to Z$ birational. By assumption,
$E$ lies in the fiber $T_0$ of $T$ over the closed point  in $Z$. Note that $T_0$ is a scheme of finite type over $K$, where $K=k(x)$, hence over $k$ (recall that $K$ is finite over $k$). Let $y\in E$ be a closed point. It follows from the Dimension Formula (see \cite[Theorem~15.6]{Matsumura}) that 
$\dim(\cO_{T,y})=n$. Since
$E$ is a prime divisor on $T$, we deduce that $\cO_v=\cO_{T,E}$ 
is essentially of finite type over $k$. Its residue field $k(E)$ is the fraction field of
$\cO_{E,y}$. This has dimension $(n-1)$, hence 
$k(E)$ has transcendence degree $(n-1)$ over $K$ (equivalently, over $k$).

We now show that we can find a sequence of schemes $Z_0,\ldots,Z_m$, with $Z_0=Z$ and  
each $Z_i$ being the blow-up of $Z_{i-1}$ at the center $C_{i-1}$ of $\ord_E$ on $Z_{i-1}$, such that
the center of $\ord_E$ on $Z_m$ has codimension one. This is well-known in the case of schemes of finite type over a field, and for example the proof of \cite[Lemma~2.45]{KM} can be easily adapted
to our setting. Indeed, arguing as in \emph{loc. cit.} one first shows that if the $Z_i$ are constructed as above, then $\cO_v=\bigcup_{i\geq 1}\cO_{Z_i,C_i}$. If $y_1,\ldots,y_{n-1}\in \cO_v$ are such that their residues give a transcendence basis of $k(E)$ over $K$, let $i$ be large enough such that 
all the $y_j$ lie in $\cO_{Z_i,C_i}$. Therefore the residue field  of $\cO_{Z_i,C_i}$ has transcendence degree $(n-1)$ over $K$. Another use of the Dimension Formula implies that
$\dim(\cO_{Z_i,C_i})=1$. Since $\widehat{\cO_{X,x}}$ is a Nagata ring, so is
$\cO_{Z_i,C_i}$, hence the normalization $S$ of $\cO_{Z_i,C_i}$ is finite over 
$\cO_{Z_i,C_i}$. $S$ is a Dedeking ring, and $\cO_v$ is the localization of $S$ at a maximal ideal.
If $m$ is such that $\cO_{Z_m,C_m}$ contains $S$, then we see that $\codim(C_m,Z_m)=1$.

We similarly construct a sequence of varieties $X_i$,
where $X_0=X$ and $X_i$ is the blow-up of $X_{i-1}$ along 
the center $C_{i-1}^X$ of the restriction $w$ of $\ord_E$ to the
function field of $X$. 
It follows by induction on $i$ that we have Cartesian diagrams
$$
\xymatrix{
Z_i \ar[r]^{g_i} \ar[d] & X_i \ar[d] \\
Z_{i-1} \ar[r]^{g_{i-1}} & X_{i-1} 
}
$$
such that $C_i=(g_i)^{-1}(C_i^X)$
(this follows since $C_i^X$ is clearly the closure of $g_i(C_i)$, 
and since $C_i$ lies over the closed point in $Z$). Since $g_m$
induces an isomorphism between $C_m$ and $C^X_m$, it follows that
$\codim(C^X_m,X_m)=1$. Therefore $w$ is a divisorial valuation. The last 
assertion in the lemma follows by taking a model $Y$ over $X$ on which 
the proper transform of $C^X_m$ is nonsingular.
\end{proof}

\begin{proof}[Proof of Proposition~\ref{semicont}]
If $k$ is algebraically closed, then
the assertion for usual ideals on $\Spec(\cO_{X,x})$ follows from
\cite[Theorem~1.4]{dFEM} (see also \cite{Kol1}). The assertion extends to the case when $k$ is not
algebraically closed  via Lemma~\ref{alg_closed}. 

We now extend this to the case of
(usual) ideals $\frA$ and $\frB$ on $Z$ (that is, we assume $r=1$ and $q_1=1$). By Lemma~\ref{divisor} below, there is a divisor 
$F$ over $X$, with center $x$, such that $\ord_F$ is equal to the restriction of $\ord_E$
to the function field of $X$. Furthermore, it follows from the lemma and Proposition~\ref{prop3_2}
that given an ideal $\frA'$ on $X$ with $x\in {\rm Supp}(\frA')$, $F$ computes
$\lct_x(\frA')$ if and only if $E$ computes $\lct(\widehat{\frA'})$.

Recall that $\frn$ is the ideal defining $x\in X$,
and $\frm$ is the ideal defining the closed point in $Z$. 
For every $d\geq 1$, after replacing $X$ by any affine neighborhood of $x$, we can find
 ideals $\frA'_d$ and $\frB'_d$ on $X$ such that
$\widehat{\frA'_d}=\frA+\frm^d$ and $\widehat{\frB'_d}=\frB+\frm^d$.
Since $\ord_E(\frm)\geq 1$, it follows that $\ord_E(\frA+\frm^d)=\ord_E(\frA)$
for $d>\ord_E(\frA)$. Let us fix such $d$. Since $E$ computes $\lct(\frA)$, 
we deduce that
$\lct(\frA+\frm^d)\leq \lct(\frA)$,
and if equality holds, then $E$ computes $\lct(\frA+\frm^d)$.
On the other hand, the inclusion $\frA\subseteq\frA+\frm^d$ implies 
$\lct(\frA)\leq\lct(\frA+\frm^d)$. We conclude that 
$\lct(\frA)=\lct(\frA+\frm^d)$, and $E$ computes $\lct(\frA+\frm^d)$. 
Therefore $F$ computes $\lct(\frA'_d)=\lct(\frA)$. Since 
in a neighborhood of $x$, the ideals $\frB'_d$ and $\frA'_d$
are congruent modulo $\{u\mid\ord_F(u)>\ord_F(\frA'_d)\}$,
we deduce from the case we already know that $\lct(\frB+\frm^d)=\lct(\frB'_d)=\lct(\frA)$. This holds for all $d>\ord_E(\frA)$, hence
it follows from Proposition~\ref{prop_limit} that $\lct(\frB)=\lct(\frA)$.

Let us deduce now the statement of the proposition in the case of $\RR$-ideals. 
We first note that the hypothesis implies that
$\ord_E(\fra_i)=\ord_E(\frb_i)$ for all $i$.

\noindent {\bf Claim}. For all positive integers $\ell_1,\ldots,\ell_r$, we have
\begin{equation}\label{eq_prop_m_adic}
\prod_{i=1}^r\frb_i^{\ell_i}+J_{\ell_1,\ldots,\ell_r}=
\prod_{i=1}^r\fra_i^{\ell_i}+J_{\ell_1,\ldots,\ell_r},
\end{equation}
where $J_{\ell_1,\ldots,\ell_r}=\{f\in \widehat{\cO_{X,x}}\mid\ord_E(f)>\ord_E(\fra_1^{\ell_1}\cdots\fra_r^{\ell_r})\}$. 

After replacing each $\fra_i$ and $\frb_i$ by $\ell_i$ copies of itself, we see that it is enough to
prove the claim when all $\ell_i=1$. If $u_i\in\frb_i$, let us write $u_i=w_i+h_i$, with
$w_i\in\fra_i$, and $h_i\in\frp_i$. In this case,
\begin{equation}\label{eq_prop_m_adic2}
\prod_{i=1}^ru_i-\prod_{i=1}^rw_i=\sum_{m=1}^r\left(h_m\cdot \prod_{i<m}w_i\cdot
\prod_{j>m}u_j\right).
\end{equation}
Since each of the terms on the right-hand side of (\ref{eq_prop_m_adic2})
has order $>\sum_m\ord_E(\fra_m)$, we deduce the inclusion ``$\subseteq$" in 
(\ref{eq_prop_m_adic}), and the opposite inclusion follows by symmetry.
This proves the claim. 

We now apply Lemma~\ref{lem_reduction2} to get sequences of positive rational numbers
$(q_{m,j})_j$ for $1\leq j\leq r$, with $\lim_{m\to\infty}q_{m,j}=\lct(\frA)\cdot q_j$ and such that,
for all $m$, 
$\lct(\prod_{i=1}^r\fra_i^{q_{m,i}})=1$ and this log canonical threshold is computed by $E$. 
We choose for each $m$ a positive integer $N_m$ such that $N_mq_{m,j}\in\ZZ$ for all $j$.
If we put $\fra^{(m)}:=\prod_{i=1}^r\fra_i^{N_mq_{m,i}}$ and
$\frb^{(m)}:=\prod_{i=1}^r\frb_i^{N_mq_{m,i}}$, then $\lct(\fra^{(m)})=1/N_m$
is computed by $E$. The above claim, together with the case we already know (for usual ideals)
gives $\lct(\frb^{(m)})=1/N_m$ for every $m$. We now deduce 
$\lct(\frB)=\lct(\frA)$ from Lemma~\ref{lem_reduction}, which completes the proof
of the proposition.
\end{proof}

\section{Generic limit constructions}

In this section we give a variant of the generic limit construction from \cite{Kol1}
(see also \cite{dFEM}), that allows us to deal with the fact that in Theorem~\ref{main}
the ambient space is not fixed. Let us fix an algebraically closed field $k$ of characteristic zero. Let
$R=k\llbracket x_1,\ldots,x_N\rrbracket$, 
let $\frm$ be the maximal ideal in $R$, and for every field extension $K/k$ let
$R_K=K\llbracket x_1,\ldots,x_N\rrbracket$ and $\frm_K=\frm R_K$.

For every $d \ge 1$, we consider the quotient homomorphism $R \to R/\frm^d$.
We identify the ideals in $R/\frm^d$ with the ideals in $R$ containing $\frm^d$.
Let $\cH_d$ be the Hilbert scheme parametrizing the ideals in $R/\frm^d$, with the
reduced scheme structure. Since $\dim_k(R/\frm^d)<\infty$, $\cH_d$ is an algebraic variety.
For every field extension $K/k$, the $K$-valued points of $\cH_d$ correspond to
ideals in $R_K/\frm_K^d$.
Mapping an ideal in $R/\frm^{d}$ to its image in $R/\frm^{d-1}$ gives a surjective map
$\tau_d\colon\cH_d\to \cH_{d-1}$. This is not a morphism. However, by generic flatness
we can cover $\cH_d$ by finitely many disjoint locally closed subsets  such that
the restriction of $\tau_d$ to each of these subsets is a morphism.

Let us fix a positive integer $m$.
We also consider a parameter space $\cG$ for ideals in $k[x_1,\ldots,x_N]$ generated 
by polynomials of degree $\leq m$, and vanishing at the origin (we consider on $\cG$ the
reduced structure). Each such ideal is determined by its intersection with the vector space of polynomials of degree $\leq m$, hence $\cG$ is an algebraic variety.
Given a field extension $K/k$, the $K$-valued points of $\cG$ are in bijection
with the ideals
in $K[x_1,\ldots,x_N]$ vanishing at the origin and generated in degree $\leq m$.

We now fix also a positive integer $r$. Consider the product
$\cZ_d:=\cG\times (\cH_d)^r$ and the map $t_d\colon\cZ_d\to\cZ_{d-1}$
that is given by the identity on the first component, and by $\tau_d$ on the other components.
As in the case of $\tau_d$, even though 
${t_d}$ is not a morphism, we can cover $\cZ_d$ by disjoint locally closed subsets
such that the restriction of $t_d$ to each of these subsets is a morphism.
In particular, 
for every irreducible closed subset
$Z\subseteq\cZ_d$, the map ${t_d}$ induces a rational map
$Z\rat \cZ_{d-1}$.

We now describe the generic limit construction. The main  difference with the treatment in \cite{dFEM}
is coming from the first factor in $\cZ_d$. 
Suppose that we have $(r+1)$ sequences
$(\frp_i)_{i\in I_0},(\fra^{(1)}_i)_{i\in I_0},\ldots,(\fra^{(r)}_i)_{i\in I_0}$ indexed by $I_0=\ZZ_+$. 
Each $\frp_i$ is an ideal in $k[x_1,\ldots,x_N]$ generated in degree
$\leq m$ and vanishing at the origin, and each $\fra_i^{(j)}$ 
is an ideal in $k\llbracket x_1,\ldots,x_N\rrbracket$.

We consider sequences of irreducible closed subsets 
$Z_d\subseteq \cZ_d$ for $d\geq 1$
with the following properties:
\begin{enumerate}
\item[$(\star)$] For every $d\geq 1$, the projection $t_{d+1}$ induces a dominant 
rational map $\phi_{d+1}\colon
Z_{d+1}\rat Z_{d}$.
\item[$(\star\star)$] For every $d\geq 1$, there are infinitely many $i$ with 
$(\frp_i, \fra^{(1)}_i+\frm^d,\ldots,\fra^{(r)}_i+\frm^d)\in Z_d$, and the set of such 
$(r+1)$-tuples is dense in $Z_d$.
\end{enumerate}
Given such a sequence $(Z_d)_{d\geq 1}$, 
we define inductively nonempty  open subsets $Z^\o_d\subseteq
Z_d$, and a nested sequence of infinite subsets
$$I_0\supseteq I_1\supseteq I_2\supseteq\cdots,$$ 
as follows. We put $Z^\o_1=Z_1$ and
$I_1=\{i\in I_0\mid(\frp_i, \fra^{(1)}_i+\frm,\ldots,\fra^{(r)}_i+\frm)\in Z_1^{\o}\}$. For $d\geq 2$, let $Z^\o_d=\phi_d^{-1}(Z^\o_{d-1})
\subseteq {\rm Domain}(\phi_d)$ and 
$I_d=\{i\in I_0\mid (\frp_i, \fra_i^{(1)}+\frm^d,\ldots
\fra_i^{(r)}+\frm^d)\in Z^\o_d\}$.
It follows by induction on $d$ that $Z^\o_d$ is open in $Z_d$, and condition $(\star\star)$ implies
that each $I_d$ is infinite. Furthermore, it is clear that $I_d\supseteq I_{d+1}$.

Sequences $(Z_d)_{d\geq 1}$ satisfying $(\star)$ and $(\star\star)$ can be constructed
as in \cite[Section 4]{dFEM}. 
Suppose now that we have a sequence $(Z_d)_{d\geq 1}$ with these two properties.
The rational maps $\phi_d$ induce a nested sequence of function fields $k(Z_d)$. Let
$K:=\bigcup_{d \ge 1} k(Z_d)$.
Each morphism $\Spec(K)\to Z_d\subseteq\cZ_d$ 
corresponds to an $(r+1)$-tuple $(\widetilde{\frp}_d,\widetilde{\fra}_d^{(1)},\ldots,
\widetilde{\fra}^{(r)}_d)$. All $\widetilde{\frp}_d$ are equal, and we denote this ideal
in $K[x_1,\ldots,x_N]$ by $\frp$. This is generated in degree
$\leq m$ and vanishes at the origin.
Furthermore, the compatibility between the morphisms $\Spec(K)\to \cZ_d$
implies that there are (unique) ideals ${\fra}^{(j)}$ in $R_K$ (for $1\leq j\leq r$) such that 
$\widetilde{\fra}^{(j)}_d={\fra}^{(j)}+\frm_K^d$ for every $d$. 
We call the $(r+1)$-tuple $(\frp,{\fra}^{(1)},\ldots,{\fra}^{(r)})$
the \emph{generic limit} of the given $(r+1)$ sequences of ideals. 

We record in the next lemma some easy properties of this construction.
The proof is straightforward, so we omit it.

\begin{lemma}\label{easy_properties}
With the above notation, for every $j$ with $1\leq j\leq r$, the following hold.
\begin{enumerate}
\item[i)] If $\fra^{(j)}_i=\frb$ for some ideal $\frb\subseteq R$ and every $i$, 
then $\fra^{(j)}=\frb R_K$.
\item[ii)] If $\fra^{(j)}_i\subseteq \frm$ for every $i$, then $\fra^{(j)}\subseteq
\frm_K$.
\item[iii)] If $\fra^{(j_1)},\ldots,\fra^{(j_s)}=(0)$, then for every $q\geq 1$ there are infinitely many $d$ such that
$\fra^{(j_{\alpha})}_d\subseteq\frm^q$ for $1\leq\alpha\leq s$.
\end{enumerate}
\end{lemma}

Let $\cI_d$ be the universal ideal on $\cH_d\times\AAA_k^N$, whose restriction 
to the fiber over a point in $\cH_d$ corresponding to the ideal $\frb$ containing $\frm^d$
is equal to $\frb$. Similarly, let $\cJ$ be the universal ideal on $\cG\times\AAA_k^N$. 
Let $\beta_j$ be the composition of the embedding $Z_d\times\AAA_k^N\hookrightarrow
{\mathcal Z}_d\times\AAA^N_k$ with the projection ${\mathcal Z}_d\times\AAA^N_k
\to \cH_d\times\AAA_k^N$ if $j \ne 0$
(or ${\mathcal Z}_d\times\AAA_k^N\to\cG\times\AAA_k^N$ when $j=0$)
induced by the projection $\cH_d^r\to\cH_d$
 onto the $j^{\rm th}$ factor (respectively by the projection $\cG\times\cH_d^r\to
 \cG$). We consider the
ideals on $Z_d\times\AAA_k^N$ defined as follows:
$\cI_d^{(0)}=(\beta_0)^{-1}(\cJ)$ and $\cI_d^{(j)}=(\beta_j)^{-1}(\cI_d)$
for $1\leq j\leq r$. 
For a not necessarily closed point $z\in Z_d$, we denote by $(\cI_d^{(j)})_z$ the restriction
of $\cI_d^{(j)}$ to the fiber over $z$. Each tuple
$(\frp_i, \fra_i^{(1)}+\frm^d,\ldots,\fra_i^{(r)}+\frm^d)$, with $i\in I_d$,
corresponds to a closed point $t_{d,i}\in Z_d$ such that 
$(\cI_d^{(j)})_{t_{d,i}}\cdot R=\fra_i^{(j)}+\frm^d$ for $1\leq j\leq r$, and
$(\cI_d^{(0)})_{t_{d,i}}=\frp_i$. By construction, the set
$T_d:=\{t_{d,i}\mid i\in I_d\}$ is dense in $Z_d$. 
Similarly, if $\eta_d\in Z_d$ is the generic point, we have
$(\cI_d^{(j)})_{\eta_d}\cdot R_K=\fra^{(j)}+\frm_K^d$ for $j\geq 1$, 
and $(\cI_d^{(0)})_{\eta_d}\cdot K[x_1,\ldots,x_N]=\frp$.
We denote by $\sigma_d\colon Z_d\to Z_d\times\AAA_k^N$ the morphism given by
$\sigma_d(u)=(u,0)$.

\begin{lemma}\label{inclusion}
With the above notation, suppose that
 $\frp_i R\subseteq\fra^{(j)}_i$ for every $i\in I_0$
and every $j$ with $1\leq j\leq r$.
For every $d\geq 1$ there is
a nonempty open subset $U_d$ of $Z_d$
such that 
$\cI_d^{(0)}\subseteq \cI_d^{(j)}$ over $U_d$ for $1\leq j\leq r$. 
In particular,
$\frp R_K\subseteq{\fra}^{(j)}$.
\end{lemma}

\begin{proof}
Since the support of $\cI_d^{(j)}$ lies in $\sigma_d(Z_d)$, it follows 
that the open subset of $Z_d\times\AAA_k^N$ where $\cI_d^{(0)}$ is contained in $\cI_d^{(j)}$
is the inverse image of an open subset $U_d$ in $Z_d$. This is nonempty, since by assumption
all $t_{d,i}$, with $i\in T_d$, lie in $U_d$. Since $\eta_d\in U_d$, we deduce that 
$\frp\subseteq (\fra^{(j)}+\frm^d)\cap K[x_1,\ldots,x_N]$. This holds for every $d$, hence
$\frp\cdot R_K\subseteq {\fra}^{(j)}$.
\end{proof}

Suppose now that the sequences $(\frp_i)_{i\in I_0},(\fra^{(1)}_i)_{i\in I_0},\ldots,
(\fra^{(r)}_i)_{i\in I_0}$ satisfy the hypothesis in Lemma~\ref{inclusion}.
In this case we put $W_i=\Spec(R/\frp_i R)$ and
$W=\Spec(R_K/\frp R_K)$.
We denote by $\ofra_i^{(j)}
=\fra^{(j)}_i/\frp_iR$ the ideal on $W_i$ corresponding to
$\fra_i^{(j)}$, for $1\leq j\leq r$.  It follows from Lemma~\ref{inclusion}
that we may consider the ideals 
$\ofra^{(j)}={\fra}^{(j)}/\frp R_K$ on $W$, for $1\leq j\leq r$. 
We denote by $\ofrm_i$ and $\ofrm_K$
 the ideals defining the closed points in $W_i$ and, respectively, in $W$.

The following is the main technical result about generic limits in our setting.

\begin{proposition}\label{key_technical}
With the above notation, if all $W_i$ are klt, and all $\fra^{(j)}_i$ are proper ideals,
then the following hold.
\begin{enumerate}
\item[i)] $W$ is klt.
\item[ii)] For every $d$ there is an infinite subset
$I_d^\o \subseteq I_d$ such that for all nonnegative real numbers $p_1,\ldots,p_r$,
and for every $i\in I_d^\o$
$$
\lct\big(W, {\textstyle\prod}_{j=1}^r\big(\ofra^{(j)} + \ofrm_K^d\big)^{p_j}\big)
 = \lct\big(W_i, {\textstyle\prod}_{j=1}^r\big(\ofra^{(j)}_i + \ofrm_i^d\big)^{p_j}\big).
$$
\item[iii)] If $E$ is a divisor over $W$ computing
$\lct(W, \prod_{j=1}^r(\ofra^{(j)})^{p_j})$ for some nonnegative real numbers 
$p_1,\ldots,p_r$, and having center equal to the closed point, then there is an integer 
$d_E$ such that
for every $d \ge d_E$ the following holds: there is an infinite subset
$I_d^E \subseteq I_d^\o$, and for every $i\in I_d^E$
a divisor $E_i$ over $W_i$ computing 
$\lct(W_i, \prod_{j=1}^r(\ofra^{(j)}_i+\ofrm_i^d)^{p_j})$, 
such that $\ord_E(\ofrm_K)=\ord_{E_i}(\ofrm_i)$ and 
$\ord_E(\ofra^{(j)}+\ofrm_K^d)=\ord_{E_i}(\ofra^{(j)}_i+\ofrm_i^d)$ for every $j$.
In particular, $E_i$ has center the closed point of $W_i$.
\end{enumerate}
\end{proposition}

We emphasize that in part iii) of the proposition, both $d_E$ and the sets $I_d^E$
depend on $p_1,\ldots,p_r$, and $E_i$ depends on $d$.

\begin{proof}
With the notation in Lemma~\ref{inclusion}, let
$X_d$ be the closed subscheme of $U_d\times\AAA_k^N$ defined by 
$\cI_d^{(0)}$, and let $f\colon X_d\to U_d$ be the morphism induced by projection. 
It follows from the Lemma~\ref{inclusion} that we can define ideal sheaves 
$\frb^{(j)}_d$ on  $X_d$ as the quotient of (the restrictions to $X_d$ of)
$\cI_d^{(j)}$ by $\cI_d^{(0)}$. We denote by $(X_d)_{\xi}$ the fiber of $X_d$ over a (not necessarily
closed) point $\xi\in U_d$, and by $(\frb^{(j)}_d)_{\xi}$ the restriction of 
$\frb^{(j)}_d$ to $(X_d)_{\xi}$.
We will apply Corollary~\ref{cor_for_key_technical} to $f$
and to the ideals $\frb^{(j)}_d$, with $1\leq j\leq r$.

Note that each $W_i$ is obtained by completing at $\sigma_d(t_{d,i})$ the fiber
$(X_d)_{t_{d,i}}$. Since $W_i$ is klt, it follows from Proposition~\ref{prop3_2}
that $(X_d)_{t_{d,i}}$ is klt at $\sigma_d(t_{d,i})$. Corollary~\ref{cor_for_key_technical}
implies that $(X_d)_{\eta_d}$ is klt at $\sigma_d(\eta_d)$. Therefore the base-extension 
to $K$ of this generic fiber is klt at the origin (see Lemma~\ref{alg_closed}), hence its completion $W$ is klt by another application of Proposition~\ref{prop3_2}. This proves i).

The assertion in ii) follows directly from Proposition~\ref{prop3_2} and Corollary~\ref{cor_for_key_technical} (see also Remark~\ref{independence_of_exponents}). 
In order to prove iii), note first that if $d_E>\ord_E(\ofra^{(j)})$ for all $j$ with $1\leq j\leq r$, then
$E$ also computes 
$$\lct\big(W, {\textstyle\prod}_j(\ofra^{(j)}+\ofrm_K^d)^{p_j}\big)=
\lct\big(W, {\textstyle\prod}_j(\ofra^{(j)})^{p_j}\big)$$
for all $d\geq d_E$
(the inequality ``$\leq$" follows since the assumption on $d$ implies
$\ord_E(\ofra^{(j)}+\ofrm_K^d)=\ord_E(\ofra^{(j)})$ for all $j$, while
the reverse inequality follows from the inclusions
$\ofra^{(j)}\subseteq
\ofra^{(j)}+\ofrm_K^d$). 
By Lemma~\ref{divisor}, the restriction of the valuation
$\ord_E$ to the fraction field of $K[x_1,\ldots,x_N]$ is equal to $\ord_F$, for some divisor $F$
over $\AAA_K^N$. 
We can choose $d_E$ large enough, so that 
for $d\geq d_E$ the divisor $F$ comes by base extension from 
a divisor $F_d$ defined over
$k(Z_d)$. It follows from Proposition~\ref{prop3_2} and Lemma~\ref{alg_closed}
that  $F_d$ computes $\lct_{\sigma_d(\eta_d)}((X_d)_{\eta_d}, 
(\prod_j\frb_d^{(j)})^{p_j}_{\eta_d})$.
The assertion in iii) now follows from Corollary~\ref{cor_for_key_technical} ii).
This completes the proof of the proposition.
\end{proof}

\begin{corollary}\label{limit_lct}
With the notation and assumption in Proposition~\ref{key_technical}, for every sequence $(i_d)_{d \ge 1}$ with
$i_d \in I_d^\o$, and for all nonnegative real numbers $p_1,\ldots,p_r$
we have 
$$\lct\big(W,{\textstyle\prod}_{j=1}^r(\ofra^{(j)})^{p_j}\big) = 
\lim_{d \to \infty} \lct\big(W_{i_d}, {\textstyle\prod}_{j=1}^r(\ofra^{(j)}_{i_d})^{p_j}\big).$$ 
In particular, if the sequence 
$\big(\lct(W_i,{\textstyle\prod}_{j=1}^r(\ofra^{(j)}_{i})^{p_j}\big)_{i \ge 1}$ is convergent, then it
converges to $\lct(W,{\textstyle\prod}_{j=1}^r(\ofra^{(j)})^{p_j})$.
\end{corollary}

\begin{proof}
Since $W$ and all $W_i$ are klt, 
it follows from Proposition~\ref{prop3_9} ii) that given any $\epsilon>0$, if $d\geq \frac{N}{\epsilon p_j}$
for every $j$ such that $p_j>0$, then we have
$$\big|\lct\big(W,{\textstyle\prod}_j\ofra^{(j)})^{p_j}\big) -
\lct\big(W, {\textstyle\prod}_j(\ofra^{(j)} + \ofrm_K^d)^{p_j}\big)\big|\leq\epsilon,$$
$$\big|\lct\big(W_{i_d}, {\textstyle\prod}_j(\ofra^{(j)}_{i_d})^{p_j}\big)-
\lct\big(W_{i_d}, {\textstyle\prod}_j(\ofra^{(j)}_{i_d} + \ofrm_{i_d}^d)^{p_j}
\big)\big|
\leq\epsilon.$$
It follows from the choice of $I_d^\o$ in Proposition~\ref{key_technical} that
$$\big|\lct\big(W,{\textstyle\prod}_j(\ofra^{(j)})^{p_j}\big)-
\lct\big(W_{i_d}, {\textstyle\prod}_j(\ofra^{(j)}_{i_d})^{p_j}\big)\big|\leq 2\epsilon.$$
This gives the assertion in the corollary.
\end{proof}

\begin{remark}\label{rem_some_ideal_zero}
The argument in the proof of the above corollary can also be carried out if some $\ofra^{(j)}$ is zero;
in this case, one has to apply part i) in Proposition~\ref{prop3_9}, instead of part ii).
In particular, we see that if the sequence $\Big(\lct\big(W_i,{\textstyle\prod}_j(\ofra^{(j)}_{i})^{p_j}\big)\Big)_{i \ge 1}$ converges to a positive 
real number, then all $\ofra^{(j)}$, with $1\leq j\leq r$ are nonzero.
\end{remark}

\section{Proof of the main result}

Our goal is to give a proof of Theorem~\ref{main}. 
Let us fix an algebraically closed ground field $k$ of characteristic zero.
Consider a collection $(X_i,x_i)$ of schemes of finite type over $k$, with
$x_i\in X_i$ closed points. We say that the family 
has bounded
singularities if there are positive integers $m$ and $N$ such that for every $i$
there is a closed subscheme $Y_i$ of $\AAA_k^N$ whose ideal is defined by
polynomials of degree $\leq m$, and a closed point $y_i\in Y_i$ such that 
$\widehat{\cO_{X_i,x_i}}\simeq\widehat{\cO_{Y_i,y_i}}$. 

\begin{remark}
The above condition is equivalent with the existence of a morphism 
$\pi\colon {\mathcal Y}\to T$ 
of schemes of finite type over $k$ such that for every $i$ there 
is a closed point $t_i\in T$ and a closed point $y_i$ in the fiber ${\mathcal Y}_{t_i}$ over $t_i$
such that $\widehat{\cO_{X_i,x_i}}\simeq\widehat{\cO_{{\mathcal Y}_{t_i},y_i}}$.
Indeed, if the collection of varieties has bounded singularities,
then it is enough to take $T$ to be a parameter 
space parametrizing closed subschemes of $\AAA_k^N$ defined by ideals generated in degree $\leq m$,
and let ${\mathcal Y}\hookrightarrow\AAA_k^N\times T$ 
be the universal subscheme. Conversely, given
$\pi$ we can find finite affine open covers $T=\bigcup_j V_j$ and ${\mathcal Y}=\bigcup_jU_j$
such that $\pi(U_j)\subseteq V_j$ for every $j$. It is enough to take $N$ and $m$ such that 
each $U_j$ can be embedded as a closed subscheme of $\AAA_k^N$, with the ideal generated
in degree $\leq m$.
\end{remark}

We now set the notation for the rest of this section. Let us fix $N$ and $m$, and let
${\mathcal X}_{N,m}$ be the set of all klt schemes of the form
$\Spec(\widehat{\cO_{X,x}})$, where $X$ is a closed subscheme of $\AAA_k^N$
defined by an ideal generated by polynomials of degree $\leq m$ and $x\in X$
is a closed point. After a suitable translation 
we may always assume that $x$ is the origin in $\AAA_k^N$. We will freely use the notation
introduced in the previous section in the construction of generic limits.
The following is our main result.

\begin{theorem}\label{new_main}
With the above notation, if $\Gamma\subset\RR_+$ is a DCC subset, then
there is no infinite strictly increasing sequence of log canonical thresholds
$$\lct_{(W_1,\frB_1)}(\frA_1)<\lct_{(W_2,\frB_2)}(\frA_2)<\cdots,$$
where all $W_i\in {\mathcal X}_{N,m}$, and $\frA_i$, $\frB_i$ are $\Gamma$-ideals
on $W_i$, with $(W_i,\frB_i)$ log canonical.
\end{theorem}

\begin{proof}
Let us assume that there is a strictly increasing sequence as in the statement,
with $W_i=\Spec(\widehat{\cO_{X_i,0}})$ klt, where $X_i$ is a closed subscheme 
of $\AAA_k^N$ defined by an ideal $\frp_i\subset k[x_1,\ldots,x_N]$ generated
in degree $\leq m$. Let $c_i=\lct_{(W_i,\frB_i)}(\frA_i)$. 

Let us write $\frA_i=\prod_{j=1}^{r_i}(\ofra_i^{(j)})^{p_{i,j}}$, 
and $\frB_i=\prod_{j=1}^{s_i}(\ofrb_i^{(j)})^{q_{i,j}}$
where all $p_{i,j}, q_{i,j}\in\Gamma$.
We may and will assume that all $\ofra_i^{(j)}$ 
and $\ofrb_i^{(j)}$
vanish at $0$ (though it may happen that 
$\frB_i=\cO_{W_i}$, in which case $s_i=0$). 
Let $\fra^{(j)}_i$ and $\frb^{(j)}_i$ be the ideals in $R=k\llbracket x_1,\ldots,x_n\rrbracket$
such that $\ofra^{(j)}_i=\fra^{(j)}_i/\frp_iR$ and $\ofrb^{(j)}_i=
\frb^{(j)}_i/\frp_iR$.

Since $\Gamma$ satisfies DCC, it follows that there is $\epsilon>0$ such that 
$p_{i,j}, q_{i,j}>\epsilon$ for every $i$ and $j$. Since $W_i$ is klt around $0$, 
it follows from Proposition~\ref{prop3_9} that 
$$c_i\leq\lct(W_i, \frA_i)\leq\frac{\dim(W_i)}{\sum_{j=1}^{r_i}p_{i,j}}\leq \frac{N}{\epsilon r_i}.$$
This clearly implies that the sequence $(c_i)_{i\geq 1}$ is bounded, and therefore it converges to some
$c\in\RR_{>0}$. It also implies that the sequence $(r_i)_{i\geq 1}$ is bounded. Therefore, after passing
to a subsequence we may assume that $r_i=r$ for all $i$. 
Similarly, since $\lct(W_i,\frB_i)\geq 1$ for every $i$, it follows that we may assume that
$s_i=s$ for every $i$.

Using again that $\Gamma$ is a DCC set, we see that after passing to a subsequence $r+s$
times, we may assume that each of the sequences 
$(p_{i,j})_{i\geq 1}$ and $(q_{i,j})_{i\geq 1}$ is non-decreasing. 
Recall that $c_i\leq N/p_{i,j}$ 
and $1\leq N/q_{i,j}$ for every $i$ and $j$, hence the
sequences $(p_{i,j})_{i\geq 1}$ and $(q_{i,j})_{i\geq 1}$ are bounded.
We put $p_j=\lim_{i\to\infty}p_{i,j}$ and $q_j=\lim_{i\to\infty}q_{i,j}$.

Let $c'_i:=\lct_{(W_i,\frB')}(\frA'_i)$, where 
$\frA'_i=\prod_{j=1}^{r}(\ofra_i^{(j)})^{p_j}$
and $\frB'_i=\prod_{j=1}^{s}(\ofrb_i^{(j)})^{q_j}$.
Since 
$p_{i,j}\leq p_j$ and $q_{i,j}\leq q_j$ for every $i$ and  $j$, it follows that $c'_i\leq c_i$ for every $i$.
On the other hand, for every $\eta\in (0,1)$, we have $p_{i,j}/p_j$, $q_{i,j}/q_j>1-\eta$ for all $j$
and all $i\gg 0$. This implies $c_i\leq c'_i/(1-\eta)$ for all $i\gg 0$. We deduce that 
the sequence $(c'_i)_{i\geq 1}$ contains a strictly increasing subsequence converging to $c$.
Therefore, in order to derive a contradiction, we may assume that 
$p_{i,j}=p_j$ and $q_{i,j}=q_j$ for every $i$ and $j$.

\noindent{\bf Case 1}. We first treat the case when $\frB_i=\cO_X$ for every $i$. 
The argument for this case now
follows closely the proof of
\cite[Theorem~5.1]{dFEM}, using though 
the version of generic limit construction introduced in the previous section. 
We consider the sequences of ideals 
$$(\frp_i)_{i\in I_0}, (\fra^{(1)}_i)_{i\in I_0},
\ldots,(\fra^{(r)}_i)_{i\in I_0}, (\frm)_{i\in I_0},$$
where $\frm$ is the maximal ideal in $R$. 
We choose a generic limit $(\frp,\fra^{(1)},\ldots,\fra^{(r)},\frm_K)$ constructed as in 
\S 4, with $\fra^{(j)}$ proper ideals in $R_K=K\llbracket x_1,\ldots,x_N\rrbracket$.

As in Proposition~\ref{key_technical}, we consider the scheme
$W=\Spec(R_K/\frp R_K)$, and the ideals $\ofra^{(j)}=\fra^{(j)}/(\frp R_K)$.
Note that by Remark~\ref{rem_some_ideal_zero} all $\ofra^{(j)}$ are nonzero.
Let $\frA$ be the $\RR$-ideal $\prod_{j=1}^r(\ofra^{(j)})^{p_j}$.
It follows from Corollary~\ref{limit_lct} that $c=\lct(W,\frA)$. 

By Lemma~\ref{lem_red_zero_dimensional}, we can find a nonnegative real number $t$
such that $\lct(W,\ofrm_K^t\.\frA)=c$, 
and this is computed by a divisor $E$ over $W$ having center
equal to the closed point (recall that $\ofrm_K$ defines the closed point of $W$, and $\ofrm_i$
defines the closed point on $W_i$). It follows from Proposition~\ref{key_technical} iii) that we can find $d_E$
such that for every $d\geq d_E$ the following holds: there is an infinite subset $I_d^E\subseteq
I_0$ such that for each $i\in I_d^E$ there is a divisor $E_i$ over 
$W_i$
having center equal to the closed point, computing the log canonical threshold
$\lct(W_i, \ofrm_i^t\.{\textstyle\prod}_{j=1}^r(\ofra^{(j)}_i+\ofrm_i^d)^{p_j})$, and such that 
$\ord_E(\ofra^{(j)}+\ofrm_K^d)=\ord_{E_i}(\ofra^{(j)}_i+\ofrm_i^d)$ for every $j$.

Fix $d\geq d_E$ such that, in addition, $d>\ord_E(\ofra^{(j)})$ for all $j\geq 1$. 
Since $\lct(W,\ofrm_K^t\.\frA)$ is computed by $E$, and 
$\ofra^{(j)}\subseteq\ofra^{(j)}+\ofrm_K^d$, it follows that
\begin{equation}\label{eq5_1}
\lct(W, \ofrm_K^t\.\frA)=\lct(W, \ofrm_K^t\.{\textstyle\prod}_j(\fra^{(j)}+\ofrm_K^d)^{p_j}),
\end{equation}
and the right-hand side is computed by $E$. On the other hand, it follows from 
Proposition~\ref{key_technical} ii) that we may assume that for all $i\in I_d^E$
\begin{equation}\label{eq5_2}
\lct\big(W, \ofrm_K^t\. {\textstyle\prod}_j(\ofra^{(j)} + \ofrm_K^d)^{p_j}\big)
 = \lct\big(W_i, \ofrm_i^t\. {\textstyle\prod}_j(\ofra^{(j)}_i + \ofrm_i^d)^{p_j}\big).
\end{equation}
Since $d>\ord_E(\ofra^{(j)})$, we have
\begin{equation}\label{eq5_3}
\ord_E(\ofra^{(j)})=\ord_E(\ofra^{(j)}+\ofrm_K^d)=\ord_{E_i}(\ofra_i^{(j)}+\ofrm_i^d)<d
\end{equation}
for every $i\in I_d^E$, and every $j\geq 1$. It follows from 
Proposition~\ref{semicont} that
\begin{equation}\label{eq5_4}
\lct\big(W_i, \ofrm_i^t\. {\textstyle\prod}_j(\ofra^{(j)}_i)^{p_j}\big)
=\lct\big(W_i, \ofrm_i^t\. {\textstyle\prod}_j(\ofra^{(j)}_i+\frm^d_i)^{p_j}\big).
\end{equation}
By combining equations (\ref{eq5_1}), (\ref{eq5_2}) and (\ref{eq5_4}), we conclude that
$$c=\lct\big(W_i, \ofrm^t\. {\textstyle\prod}_j(\ofra^{(j)}_i)^{p_j}\big)
\leq\lct(W_i,\frA_i)<c,$$ a contradiction. This completes the proof of this case.

\noindent{\bf Case 2}. We now treat the general case. Consider the sequences of ideals
$$(\frp_i)_{i\in I_0}, (\fra^{(1)}_i)_{i\in I_0},\ldots,(\fra^{(r)}_i)_{i\in I_0}, 
(\frb_i^{(1)})_{i\in I_0},\ldots,(\frb^{(s)}_i)_{i\in I_0}.$$
Again, we construct a generic limit
 $(\frp,\fra^{(1)},\ldots,\fra^{(r)},\frb^{(1)},\ldots,\frb^{(s)})$ as in 
\S 3. Let $W=\Spec(R_K/\frp R_K)$, and 
$\ofra^{(j)}=\fra^{(j)}/\frp R_K$ and $\ofrb^{(j)}=\frb^{(j)}/\frp R_K$.
We consider the $\RR$-ideals  $\frA=\prod_{j=1}^{r}(\ofra^{(j)})^{p_{j}}$
and $\frB=\prod_{j=1}^{s}(\ofrb^{(j)})^{q_{j}}$.

For every $c'<c$, we have $c_i>c'$ for $i\gg 0$. Therefore
$\lct(W_i,\frB_i\cdot\frA_i^{c'})\geq 1$ for such $i$. By  Proposition~\ref{key_technical} ii),
$\lct(W,\frB\cdot\frA^{c'})$ is a limit point of the sequence $(\lct(W_i,\frB_i\cdot\frA_i^{c'}))_{i\geq 1}$,
hence $\lct(W,\frB\cdot\frA^{c'})\geq 1$. Since this holds for every $c'<c$, we have
$\lct(W,\frB\cdot\frA^{c})\geq 1$.

 Another application of Proposition~\ref{key_technical}
ii) gives that $\lct(W,\frB\cdot\frA^{c})$ is a limit point of the sequence 
$(\lct(W_i,\frB_i\cdot\frA_i^{c}))_{i\geq 1}$.
On the other hand, it follows from Case 1 that the set
$\{\lct(W_i,\frB_i\cdot\frA_i^{c})\mid i\geq 1\}$ contains no strictly increasing sequences. We deduce that there are infinitely
many $i$ such that
$\lct(W_i,\frB_i\cdot\frA_i^{c})\geq \lct(W,\frB\cdot\frA^{c})\geq 1$. For every such $i$ we have
$c_i\geq c$, a contradiction. This completes the proof of the theorem.
\end{proof}

\begin{remark}
It follows from Proposition~\ref{prop3_2} that the statement in the above theorem implies the
version stated in the Introduction in Theorem~\ref{main} in terms of bounded families
of singularities. 
\end{remark}

\appendix

\section{Sheaves of differentials for schemes of finite type over a formal power series ring}

In this appendix we work in the following setting. Let $k$ be a field of characteristic zero,
and $R=k\llbracket x_1,\ldots,x_n\rrbracket$. 
All our schemes will be of finite type over such a formal power series ring. 
We note that since $R$ is an excellent ring 
(see \cite[p. 260]{Matsumura}), it follows that the nonsingular locus of such a scheme is open.
Furthermore, $R$ is universally catenary.
The usual sheaves of differentials over $k$ are not the right objects in our setting
(in particular, they are not coherent).
Our aim in this section is to introduce an appropriate version of sheaves of differentials
that is better behaved.

For every $R$-module $M$, the \emph{special $k$-derivations} $D\colon R\to M$ 
are those $k$-derivations with the property that $D(f)=\sum_{i=1}^n\frac{\partial f}{\partial x_i}D(x_i)$
for every $f\in R$.
Note that this is automatically true for a $k$-derivation if $M$ is separated in the
$(x_1,\ldots,x_n)$-adic topology, but not in general. 

For an $R$-algebra $A$ and an $A$-module $M$, the module ${\rm Der}_k'(A,M)$
of \emph{special $k$-derivations} 
consists of all $k$-derivations $D\colon A\to M$ such that the restriction of $D$ to $R$ is a special
$k$-derivation $R\to M$. It is clear that ${\rm Der}'_k(A,M)$ is an $A$-submodule of
${\rm Der}_k(A,M)$. Note that the definition depends on the fixed ring $R$.

If $w\colon M\to N$ is a morphism of $A$-modules, composing with $w$ induces a morphism
of $A$-modules ${\rm Der}'_k(A,M)\to {\rm Der}'_k(A,N)$.
We say that $A$ has a module of \emph{special differentials} if there is an $A$-module
$\Omega'_{A/k}$ with a special $k$-derivation $d'_{A/k}\colon A\to\Omega'_{A/k}$
that induces an isomorphism of $A$-modules
$$\Hom_A(\Omega'_{A/k},M)\to {\rm Der}'_k(A,M)$$
for every $A$-module $M$. Of course, in this case $\Omega'_{A/k}$ is called the
{\it module of special differentials} (note that it is unique, up to a canonical isomorphism
commuting with $d'_{A/k}$).  In order to avoid cluttering the notation we do not
include
$R$ in the notation for $\Omega'_{A/k}$. However, the reader should keep in mind that the definition
was made in reference to a fixed $R$.

\begin{lemma}\label{lem1}
If $A\to B=A/I$ is a surjective morphism of $R$-algebras, and if $\Omega'_{A/k}$ exists,
then $\Omega'_{B/k}$ exists, and we have an exact sequence
$$I/I^2\overset{\delta}\to \Omega'_{A/k}\otimes_AB\overset{u}\to \Omega'_{B/k}\to 
0,$$
where $\delta(f)=d'_{A/k}(f)\otimes 1$ and $u(d'_{A/k}(f)\otimes 1)=d'_{B/k}(f)$.
\end{lemma}

\begin{proof}
The assertion follows as in the case of usual differentials from the fact that the corresponding
sequence of $B$-modules
$$0\to {\rm Der}'_k(B,M)\to {\rm Der}'_k(A,M)\to \Hom_A(I/I^2,M)$$
is exact for every $A$-module $M$.
\end{proof}

\begin{lemma}\label{lem2}
The module of special differentials $\Omega'_{R/k}$ exists, and it is a free $R$-module
with basis $d'_{R/k}(x_1),\ldots,d'_{R/k}(x_n)$.
\end{lemma}

\begin{proof}
The assertion follows from the fact that by definition, every $D\in {\rm Der}'_k(R,M)$ is uniquely determined by
the $D(x_i)$, which can be chosen arbitrarily.
\end{proof}

\begin{lemma}\label{lem2_part2}
Let $S$ be an $R$-algebra, and $A=S[y_i\mid i\in I]$ a polynomial ring
over $S$. If $\Omega'_{S/k}$ exists, then $\Omega'_{A/k}$ exists and it is isomorphic
to the direct sum of $\Omega'_{S/k}\otimes_SA$ and a free $A$-module with basis
$\{d'_{A/k}(y_i)\mid i\in I\}$.
\end{lemma}

\begin{proof}
The assertion follows from the fact that every $D\in {\rm Der}'_k(A,M)$ is uniquely determined by
$D\vert_S$ and by the $D(y_i)$, which can be chosen arbitrarily. 
\end{proof} 

By combining the above three lemmas we deduce the following existence result.

\begin{corollary}\label{cor1}
For every $R$-algebra $A$, the module $\Omega'_{A/k}$ exists. Furthermore, if $A$
is of finite type over $R$, then $\Omega'_{A/k}$ is finitely generated over $A$.
\end{corollary}

\begin{remark}\label{comp_usual_diff}
Since ${\rm Der}'_k(A,M)\subseteq {\rm Der}_k(A,M)$ for every $A$-module $M$, it follows that we have a surjective morphism
$\Omega_{A/k}\to \Omega'_{A/k}$. In particular, 
$\Omega'_{A/k}$ is generated as an $A$-module by $\{d'_{A/k}(a)\mid
a\in A\}$.
\end{remark}

\begin{lemma}\label{lem3}
If $\phi\colon A\to B$ is a morphism of $R$-algebras, then we have an exact
sequence 
$$\Omega'_{A/k}\otimes_AB\overset{u}\to \Omega'_{B/k}
\overset{v}\to\Omega_{B/A}\to 0,$$
where $u(d'_{A/k}(f)\otimes 1)=d'_{B/k}(f)$ and $v(d'_{B/k}(f))=d_{B/A}(f)$.
\end{lemma}

\begin{proof}
The assertion follows as in the case of usual derivations from the fact that
for every $B$-module $M$, the corresponding sequence
$$0\to {\rm Der}_A(B,M)\to {\rm Der}'_k(B,M)\to {\rm Der}'_k(A,M)$$
is exact (note that if $D\in {\rm Der}_A(B,M)$, then $D$ is trivially a special $k$-derivation,
since its restriction to $R$ is zero). 
\end{proof}

\begin{lemma}\label{lem4}
Let $A$ be an $R$-algebra.
If $S$ is a multiplicative system in $A$, then we have a canonical isomorphism
$S^{-1}\Omega'_{A/k}\simeq\Omega'_{S^{-1}A/k}$.
\end{lemma}

\begin{proof}
It is enough to note that for every $S^{-1}A$-module $M$, we have canonical isomorphisms
\begin{equation}\label{eq_lem4}
\Hom_{S^{-1}A}(S^{-1}\Omega'_{A/k},M)\simeq \Hom_A(\Omega'_{A/k},M)
\simeq {\rm Der}'_k(A,M)\simeq {\rm Der}'_k(S^{-1}A,M).
\end{equation}
The last isomorphism follows from the fact that every $k$-derivation $D\colon A\to M$
admits a unique extension $\widetilde{D}\colon S^{-1}A\to M$, and it is clear that $D$
is special if and only if $\widetilde{D}$ is. The assertion in the lemma follows from (\ref{eq_lem4}).
\end{proof}

The case when $A$ is regular will play an important role. We show that in this case
$\Omega'_{A/k}$ is locally free. In fact, we have the following more precise result.

\begin{proposition}\label{prop0}
If $A$ is an algebra of finite type over $R$, and if $\frq\in {\rm Spec}(A)$ is such that
$A_{\frq}$ is a regular ring of dimension $r$, then $\Omega'_{A_{\frq}/k}$ is a free
$A_{\frq}$-module, of rank equal to $r+\dim_{k(\frq)}(\Omega'_{k(\frq)/k})$, where
$k(\frq)=A_{\frq}/\frq A_{\frq}$. Furthermore,
if $u_1,\ldots,u_r\in A$ induce a regular system of parameters
in $A_{\frq}$, then the images of $d'_{A/k}(u_1),\ldots,d'_{A/k}(u_r)$ in $\Omega'_{A_{\frq}/k}$
are part of a basis.
\end{proposition}

\begin{proof}
Our argument is based on the results in \cite{Matsumura2}.
Since the regular locus of $A$ is open, we may replace $A$ by a localization $A_f$
in order to assume that $A$ is a regular ring, and further, that it is a domain. Let us choose
an isomorphism $A\simeq S/P$, with $S=R[y_1,\ldots,y_N]$ and $P\in \Spec(S)$. Let $Q\in {\rm Spec}(S)$
be such that $\frq=Q/P$. We put $s={\rm codim}(P)$, so that ${\rm codim}(Q)=r+s$. 

A key property of $S$ is that it satisfies the \emph{strong Jacobian
condition over} $k$. At the prime $Q$, this means that there are $D_1,\ldots,D_{r+s}
\in {\rm Der}_k(S)\cdot S_Q$, and $w_1,\ldots,w_{r+s}\in Q$, such that 
${\rm det}(D_i(w_j))\not\in QS_Q$ and $[D_i,D_j]\in \sum_{\ell=1}^{r+s} S_Q\cdot D_{\ell}$. 
The fact that $S$ satisfies this condition follows from \cite[Theorem~6]{Matsumura2}, 
which says that
rings with the strong Jacobian condition are closed under taking polynomial or formal power series rings (the statement therein is in terms of absolute derivations, but the proof goes through if one works with $k$-derivations). In this case \cite[Theorem~5]{Matsumura2} implies that 
$S_Q/PS_Q$ is regular if and only if there are $w'_1,\ldots,w'_s\in P$ such that some
$s$-minor of the matrix $(D_i(w'_j))_{i,j}$ does not lie in $QS_Q$.

Lemma~\ref{lem1} gives (after localizing and using Lemma~\ref{lem4}) an exact sequence of $A_{\frq}$-modules
$$PS_Q/P^2S_Q\overset{\phi}\to \Omega'_{S/k}\otimes_SA_{\frq}\to\Omega'_{A_{\frq}/k}\to 0.$$
Since $S_Q/PS_Q$ is regular, $PS_Q/P^2S_Q$ is a free $A_{\frq}$-module of rank $s$.
Note that since $S$ is separated with respect to the $(x_1,\ldots,x_n)$-adic topology, we have
${\rm Der}'_k(S)={\rm Der}_k(S)$. Therefore $D_1,\ldots,D_{r+s}$ define an 
$A_{\frq}$-linear map  $\psi\colon \Omega'_{S/k}\otimes_SA_{\frq}\to A_{\frq}^{r+s}$ such that
$\psi\circ\phi$ is split injective. It follows from the above exact sequence that 
$\Omega'_{A_{\frq}/k}$ is a free $A_{\frq}$-module of rank $n+N-s$, and we also see that
$w'_1,\ldots,w'_s$ generate $PS_Q$. 
Running the same argument with $P$ replaced by $Q$, we see that 
$\dim_{k(\frq)}(\Omega'_{k(\frq)/k})=n+N-(r+s)$, which gives the assertion about 
the rank of $\Omega'_{A_{\frq}/k}$.

For the last assertion in the proposition, note that
if the $\widetilde{u}_i\in S$  are lifts of the $u_i$, then it follows that $\widetilde{u}_1,\ldots,
\widetilde{u}_r,
w'_1,\ldots,w'_s$ give a minimal system of generators of $QS_Q$. By writing $w_1,\ldots,w_{r+s}$
in terms of this system of generators, we see that 
$$d'_{S/k}(\widetilde{u}_1)\otimes 1,\ldots,
d'_{S/k}(\widetilde{u}_r)\otimes 1,d'_{S/k}(w'_1)\otimes 1,\ldots,d'_{S/k}(w'_s)\otimes 1$$ are part of a basis
of $\Omega'_{S/k}\otimes_S A_{\frq}$. We deduce from this the last assertion in the proposition.
\end{proof}

For future reference, we include here the following lemma, describing the behavior
of the module of special differentials with respect to field extensions.

\begin{lemma}\label{lem2_6}
Let $K\hookrightarrow L$ be a field extension, where $K$ is an $R$-algebra.
In this case we have an exact sequence
$$0\to \Omega'_{K/k}\otimes_KL\to\Omega'_{L/k}\to\Omega_{L/K}\to 0.$$
\end{lemma}

\begin{proof}
By Lemma~\ref{lem3}, it is enough to show that the morphism
$\Omega'_{K/k}\otimes_KL\to\Omega'_{L/k}$ is injective;
equivalently, for every $L$-module $M$, the map
${\rm Der}'_k(L,M)\to {\rm Der}'_k(K,M)$ is surjective. This follows from the
fact that the map ${\rm Der}_k(L,M)\to {\rm Der}_k(K,M)$ is surjective
(recall that ${\rm char}(k)=0$), by noticing
that a $k$-derivation $D\colon L\to M$ is special if and only if $D\vert_K$ is special.
\end{proof}

We will need a comparison between the usual module of differentials for schemes of finite type over a field, and the module of special differentials for the completion at a closed point. 
We consider the following
setting. Suppose that $\phi\colon A\to B$ is a morphism of finitely generated $k$-algebras,
and $\frm$ is a maximal ideal in $A$. The field $K=A/\frm$ is a finite extension of $k$.
We put $A'=\widehat{A_{\frm}}$ and $B'=B\otimes_A\widehat{A_{\frm}}$.
By Cohen's Structure Theorem, we can find a surjective 
 local morphism of $k$-algebras
$\psi\colon S=K\llbracket x_1,\ldots,x_N\rrbracket \to A'$. In this case we take 
$R=k\llbracket x_1,\ldots,x_N\rrbracket$ and $\frm_R=(x_1,\ldots,x_N)$. Note that 
$A'$
becomes naturally an $R$-algebra via the inclusion $R\hookrightarrow S$.
We use this structure when considering the modules of special derivations
$\Omega'_{A'/k}$ and $\Omega'_{B'/k}$
Since $K/k$ is finite, $A'$ is finite as an $R$-algebra, hence $B'$ is a finitely generated 
$R$-algebra.

\begin{proposition}\label{prop1}
With the above notation, we have a canonical isomorphism
\begin{equation}\label{eq0_prop1}
\Omega_{B/k}\otimes_BB'\simeq\Omega'_{B'/k}.
\end{equation}
\end{proposition}

\begin{proof}
We first show that $\Omega_{A/k}\otimes_AA'\simeq\Omega'_{A'/k}$. 
Note that if $M$ is a finitely generated $A'$-module, then $M$ is separated with respect to
the $\frm_{A'}$-adic topology (where $\frm_{A'}$ is the maximal ideal of $A'$), hence it is separated with respect to the $\frm_R$-adic topology. Therefore
${\rm Der}'_k(A',M)={\rm Der}_k(A',M)$. Let $d_{A/k}\colon A\to \Omega_{A/k}$ be the universal
$k$-derivation on $A$, and let $j\colon \Omega_{A/k}\to\Omega_{A/k}\otimes_AA'$
be the canonical map. Note that $\Omega_{A/k}\otimes_AA'$ is complete in the 
$\frm A_{\frm}$-adic topology, hence we get a unique $\widehat{d_{A/k}}
\in {\rm Der}_k(A',\Omega_{A/k}\otimes_AA')$ such that 
$\widehat{d_{A/k}}\circ\iota = j\circ d_{A/k}$, 
where $\iota\colon A\to A'$ is the completion map.
Since $\Omega_{A/k}\otimes_AA'$
is a finitely generated $A'$-module, it follows that 
 $\widehat{d_{A/k}}$ is a special derivation, and we have
a unique morphism of $A$-modules $f\colon \Omega'_{A'/k}\to\Omega_{A/k}\otimes_AA'$
such that $f\circ d'_{A'/k}=\widehat{d_{A/k}}$.

On the other hand, since $d'_{A'/k}$ is a derivation, there is a unique morphism of
$A$-modules $g\colon \Omega_{A/k}\to\Omega'_{A'/k}$ such that $g\circ d_{A/k}=d'_{A/k}\circ\iota$.
Since $\Omega'_{A'/k}$ is finitely generated over $A'$ by Corollary~\ref{cor1}, 
it is complete, hence $g$ induces
a (unique) morphism of $A'$-modules $\widehat{g}\colon\widehat{\Omega_{A/k}}=
\Omega_{A/k}\otimes_AA'\to
\Omega'_{A'/k}$ such that $\widehat{g}\circ j=g$. It is now easy to check that
$f$ and $\widehat{g}$ are inverse isomorphisms.

In order to prove the general statement for $B$, let us write $B\simeq A[y_1,\ldots,y_m]/I$.
In this case we have
\begin{equation}\label{eq_prop1}
\Omega_{B/k}\simeq\Big((\Omega_{A/k}\otimes_AB)\oplus\bigoplus_{i=1}^mB\cdot d_{B/k}(y_i)\Big)
/B\cdot \{d_{B/k}(u)\mid u\in I\},
\end{equation}
which after tensoring with $B'$ gives a description of $\Omega_{B/k}\otimes_BB'$.
Since we have a corresponding isomorphism $B'\simeq A'[y_1,\ldots,y_m]/(\iota(I))$,  using Lemmas~\ref{lem1} and \ref{lem2_part2}
we get an analogous formula for 
$\Omega'_{B'/k}$. The isomorphism (\ref{eq0_prop1}) now follows from the corresponding
isomorphism in the case $B=A$.
\end{proof}

It is standard to deduce from Lemma~\ref{lem4}  that for every scheme $X$ over $R$,
there is a quasicoherent sheaf $\Omega'_{X/k}$ such that for every affine open subset $U$
of $X$, the restriction of $\Omega'_{X/k}$ to $U$ is canonically isomorphic to the sheaf
associated to $\Omega'_{\cO(U)/k}$. It follows from Corollary~\ref{cor1} that if $X$ is of finite type
over $R$, then $\Omega'_{X/k}$ is coherent. Furthermore, Proposition~\ref{prop0}
implies that if $X$ is nonsingular, then $\Omega'_{X/k}$ is locally free.
The exact sequences in Lemmas~\ref{lem1} and
\ref{lem3} globalize in a straightforward way.

We now use the sheaves of special differentials to introduce the notion of relative 
canonical class in this setting. 
Let $X$ be a normal scheme of finite typer over $R$. Since the discussion that follows can be
done separately on each connected component of $X$, we may and will assume that $X$ is
irreducible.
Recall that since $R$ is excellent, the nonsingular locus $X_{\rm reg}$ of $X$ is an open subset
of $X$. Since $X$ is normal, the complement $X\smallsetminus X_{\rm reg}$ has codimension 
$\geq 2$ in $X$. In particular, restriction induces an isomorphism of class groups
${\rm Cl}(X)\simeq {\rm Cl}(X_{\rm reg})$. 

The restriction $\Omega'_{X/k}\vert_{X_{\rm reg}}$ is locally free, and let $M$ be its rank. 
On $X$ we have a Weil divisor $K_X$, uniquely defined up to rational equivalence, such that
$\cO(K_X)\vert_{X_{\rm reg}}\simeq\wedge^M\Omega'_{X_{\rm reg}/k}$. As in the case of schemes
of finite type over a field, we say that $X$ is $\QQ$-Gorenstein if there is a positive integer $r$
such that $rK_X$
is a Cartier divisor (the smallest such $r$ is the \emph{index} of $X$; any other $r$ with this property
is a multiple of the index). 

Suppose now that $\pi\colon Y\to X$ is a proper birational morphism of schemes over $R$,
with $Y$ nonsingular. The following lemma shows that the relative canonical class $K_{Y/X}$
can be defined in the same way as in the case of schemes of finite type over a field
(see \cite{Kol2}).

\begin{lemma}\label{lem3_1}
With the above notation, the following hold:
\begin{enumerate}
\item[i)] We may take $K_X=\pi_*(K_Y)$.
\item[ii)] If $rK_X$ is Cartier, then there is a unique $\QQ$-divisor
$K_{Y/X}$ supported on the exceptional locus of $\pi$ such that 
$rK_Y$ and $\pi^*(rK_X)+rK_{Y/X}$ are linearly equivalent. If $X$ is nonsingular, then
$K_{Y/X}$ is effective and its support is the exceptional locus ${\rm Exc}(\pi)$. 
\item[iii)] Suppose that $X$ is nonsingular, and that $E_1+\cdots+E_q$ is a divisor on $X$ having simple normal crossings.
If $F$ is a prime nonsingular divisor on $Y$ with corresponding valuation $\ord_F$,
and if $\ord_F(E_i)=a_i$ for every $i$, then $\ord_F(K_{Y/X})\geq a_1+\cdots+a_q-1$.
\end{enumerate}
\end{lemma}

\begin{proof}
In order to prove i), we may restrict to $X_{\rm reg}$, and therefore assume that $X$ is nonsingular.
If $y\in Y$ and $x=\pi(y)$, then Lemma~\ref{lem3} gives an exact sequence
$$U:=\Omega'_{X/k,x}\otimes\cO_{Y,y}\overset{w}\to V:=\Omega'_{Y/k,y}\to\Omega_{Y/X,y}\to 0.$$
Since $\pi$ is birational, it follows from the Dimension Formula (see
\cite[Theorem~15.6]{Matsumura}) that 
$$\dim(\cO_{Y,y})=\dim(\cO_{X,x})+{\rm trdeg}(k(y)/k(x)).$$
Since $\dim_{k(y)}(\Omega_{k(y)/k(x)})={\rm trdeg}(k(y)/k(x))$, we deduce from
Lemma~\ref{lem2_6} and Proposition~\ref{prop0} that $U$ and $V$ are free
$\cO_{Y,y}$-modules of the same rank $M$. It follows that $\wedge^Mw$ is given by the
equation of an effective divisor $K_{Y/X}$. The support of this divisor is the locus where $\pi$
is not \'{e}tale, which in this case is precisely the exceptional locus of $\pi$. The assertions in 
i) and ii) now easily follow. Due to the last assertion in Proposition~\ref{prop0}, we can deduce iii)
via the same computation as in the usual case of schemes of finite type over a field.
\end{proof}

\begin{remark}\label{rem3_1}
It follows from the above proof that if $X$ is nonsingular, then $K_{Y/X}$ is independent
of the structure of $X$ as an $R$-scheme. Indeed, $K_{Y/X}$ is the effective divisor defined by
the $0^{\rm th}$ Fitting ideal of $\Omega_{Y/X}$.
It is not clear to us whether the same remains true if $X$ is singular.
\end{remark}

\begin{lemma}\label{lem3_2}
If $Y'\overset{\phi}\to Y \overset{\pi}\to X$ are proper birational morphisms, with both
$Y$ and $Y'$ nonsingular, and if $X$ is $\QQ$-Gorenstein, then 
\begin{equation}\label{eq_lem3_1}
K_{Y'/X}=K_{Y'/Y}+\phi^*(K_{Y/X}).
\end{equation} 
\end{lemma}

\begin{proof}
It is enough to observe that if $rK_X$ is Cartier, then $r(K_{Y'/Y}+\phi^*(K_{Y/X}))$ is $\pi\circ\phi$-exceptional,
and it is linearly equivalent to $rK_{Y'}-(\pi\circ\phi)^*(rK_X)$.
\end{proof}

In the next proposition we consider an integral scheme $X$, of finite type over a field $k$
(assumed, as always, to have characteristic zero).
Suppose that $\pi\colon Y\to X$ is a proper birational morphism, 
with $Y$ nonsingular. Let $x\in X$
be a closed point, and consider the Cartesian diagram
$$
\xymatrix{
W \ar[r]^h \ar[d]_f & Y \ar[d]^\pi \\
Z={\rm Spec}(\widehat{\cO_{X,x}}) \ar[r]^(.7)g & X
}
$$
(see Remark~\ref{rmk:diagram} for general properties of such a diagram).

From now on, let us assume that $X$ is normal.
In this case $W$ is connected: otherwise the fiber over the unique closed point of
$Z$ would be disconnected, but this is the same as the fiber $\pi^{-1}(x)$.
Since $\pi$ is proper and birational, we deduce that $f$ as well has these two properties.
We consider both $Z$ and $W$ as
schemes over a formal power series ring over $k$, as in Proposition~\ref{prop1}.
Since $g$ and $h$ are flat, we may pull-back Weil divisors via both $g$ and
$h$.

\begin{proposition}\label{prop3_1}
With the above notation, we may take $K_Z=g^*(K_X)$. In particular, 
$rK_X$ is Cartier in a neighborhood of $x$ if and only if $rK_Z$ is Cartier, and in this case
$h^*(K_{Y/X})=K_{W/Z}$.
\end{proposition}

\begin{proof}
Since $g$ is a regular morphism, it follows that $g^{-1}(X_{\rm reg})=Z_{\reg}$.
The first assertion in the proposition follows from the fact that if $g_0\colon {Z_{\rm reg}}\to X_{\rm reg}$
is the restriction of $g$,
then $g_0^*(\Omega_{X_{\rm reg}/k})\simeq\Omega'_{Z_{\rm reg}}$ by 
Proposition~\ref{prop1}. Furthermore, the same proposition implies that
we may take $K_W=h^*(K_Y)$.

Note now that if $D$ is a divisor on $X$, then $D$ is  Cartier in a neighborhood of $x$
if and only if $g^*(D)$ is Cartier. Indeed, for this we may assume that $D$ is effective,
and let $I=\cO(-D)\cdot\cO_{X,x}$. In this case $\cO(-g^*(D))=I\cdot\widehat{\cO_{X,x}}$,
since for every prime ideal $P$ in $\cO_{X,x}$ and every minimal prime ideal $Q$
in $\widehat{\cO_{X,x}}$ containing $P$, we have $P\cdot (\widehat{\cO_{X,x}})_Q=
Q\cdot(\widehat{\cO_{X,x}})_Q$ (this follows from the fact that the fiber over $P$
is nonsingular). It is now enough to note that $I\cdot \widehat{\cO_{X,x}}$
is principal if and only if $I$ is principal (more generally, $I$ and $I\cdot\widehat{\cO_{X,x}}$
have the same minimal number of generators). 

In particular, we see that $rK_X$ is Cartier in a neighborhood of $x$
if and only if $rK_Z$ is Cartier. 
 The last assertion in the proposition now follows from the fact that
$h^*(K_{Y/X})$ is supported on the inverse image via $h$ of the exceptional locus of $\pi$, hence
on the exceptional locus of $f$.
\end{proof}

\begin{remark}\label{divisor_over_point}
Suppose that $F$ is a prime nonsingular divisor on $Y$.
The pull-back $h^*(F)$ is a nonsingular divisor on $W$. If we consider
the irreducible components $E_1,\ldots,E_m$ of
$h^*(F)$, then 
the restriction of each $\ord_{E_i}$ to the function field of $X$ is equal to $\ord_F$.
We note that if the center of $F$ is $x$, then $h^*(F)$ 
 is abstractly isomorphic to $F$. In particular, $h^*(F)$ is a prime divisor.
\end{remark}

\section{Rational $\Q$-Gorenstein singularities in families}

Throughout this appendix, all varieties and schemes are of finite type over
a field $k$ of characteristic zero. At one point, we will need to assume that
$k=\CC$. 
Our goal is to prove Theorem~\ref{Gor_index}
on the behavior of the canonical class and Gorenstein index in families.
This implies the corollary about the generic behavior
of the log canonical threshold in families that is used in the proof of
Proposition~\ref{key_technical}.

In fact, Theorem~\ref{Gor_index} follows from the following more precise result,
for which we need to assume that $k$ is the field of complex numbers.
Given a positive integer $r$, we say that a normal variety $X$ is {\it $r$-Gorenstein}
at a point $x$ if $rK_X$ is Cartier at $x$. For a scheme $X\to T$ over $T$ we denote by 
$X_{\xi}$ the fiber over the not necessarily closed point $\xi\in T$.

\begin{theorem}\label{thm:Q-Gor}
Let $f \colon X \to T$ be a morphism of normal  varieties over $\CC$ such that
every fiber of $f$ is normal.
Then there are a positive integer $s$ and a nonempty Zariski open set $T^\o \subseteq T$
such that
for every closed point $t \in T^\o$, if $X_t$ has rational singularities at a closed point $x$,
then the following conditions are equivalent:
\begin{enumerate}
\item[(a)]
$X_t$ is $\Q$-Gorenstein at $x$;
\item[(b)]
$X_t$ is $s$-Gorenstein at $x$;
\item[(c)]
$X$ is $\Q$-Gorenstein at $x$;
\item[(d)]
$X$ is $s$-Gorenstein at $x$.
\end{enumerate}
\end{theorem}

Before giving the proof of the theorem, we start with some general considerations. 
Recall first Grothendieck's Generic Freeness Theorem 
(see, for example, \cite[Theorem~14.4]{Eis}).

\begin{theorem}[Generic Freeness Theorem]\label{thm:gen-free}
Let $\phi\colon A \to B$ be a ring homomorphism of finite type, 
with $A$ a Noetherian integral domain. 
If $M$ is a finitely generated $B$-module, then there is a nonzero $a \in A$ such that 
$M_a$ is a free $A_a$-module.
\end{theorem}

\begin{corollary}\label{cor:gen-free}
If $f \colon X \to T$ is a scheme morphism of finite type, 
with $T$ a Noetherian integral scheme, then there is a nonempty 
open subset $W$ in $T$ such that $f^{-1}(W)\to W$ is flat.
Furthermore, given a complex of coherent sheaves on $X$
$$
{\mathcal C} \colon \cF' \to \cF \to \cF'',
$$
with homology sheaf $\cH({\mathcal C})$, we can choose $W$ such
 that for every $\xi \in W$ the canonical morphism 
$\cH({\mathcal C})\otimes \cO_{X_{\xi}}  \to \cH({\mathcal C}\otimes\O_{X_{\xi}})$
is an isomorphism.
\end{corollary}

\begin{proof}
The first assertion follows easily from the theorem. For the second one, 
note that by the theorem, we may choose $W$ such that 
the images and the cokernels of the arrows in $\mathcal C$ are all flat over $W$. 
It is then easy to see that $W$ has the required property.
\end{proof}

The following lemmas will be used in the proof of Theorems~\ref{thm:Q-Gor}
and \ref{Gor_index}.

\begin{lemma}\label{lem:rest-O(mK_X)}
Let $f \colon X \to T$ be a morphism of normal schemes such that all fibers of $f$ are 
normal. For every positive integer $m$, there is an open subset $W_m \subseteq T$ 
such that for every $\xi \in W_m$ we have a canonical isomorphism 
$$
\O(mK_X)\vert_{X_{\xi}} \cong \O(mK_{X_{\xi}}).
$$ 
In particular, for every $\xi \in W_m$, the divisor $mK_X$ is Cartier at a point 
$x \in X_{\xi}$ if and only if $mK_{X_{\xi}}$ is Cartier at $x$.
\end{lemma}

\begin{proof}
By Corollary~\ref{cor:gen-free}, after replacing $T$ by an open subset
we may assume that $f$ is flat. 
We may clearly also assume that $T$ is nonsingular.
In particular, if $x$ 
is a nonsingular point of $X_{\xi}$, then both $f$ and $X$ are smooth at $x$. 
In this case we clearly have a canonical isomorphism 
$\O(mK_X)\vert_{X_{\xi}} \cong \O(mK_{X_{\xi}})$ 
in a neighborhood of $x$ (where $X_{\xi}$ is considered as a scheme over 
$\Spec(k(\xi))$. Since the complement of $(X_{\xi})_{\rm reg}$
in $X_{\xi}$ has codimension $\ge 2$, it is enough to find $W_m$ such that for every $\xi \in W_m$, 
the restriction $\O(mK_X)\vert_{X_{\xi}}$ is reflexive.

After covering $X$ by affine open subsets, we may assume that $X$ is affine. Since 
$\O(mK_X)$ is reflexive, we may write it as the kernel of a morphism 
$\phi\colon {\mathcal E_1} \to {\mathcal E}_0$ of free coherent sheaves on $X$. 
By Corollary~\ref{cor:gen-free},
there is an open subset $W_m \subseteq T$ such that for every $\xi$ in $W_m$
the restricted sheaf $\O(mK_X )\vert_{X_{\xi}}$ is isomorphic to the kernel
of the restriction of $\phi$ to $X_{\xi}$, which is a reflexive sheaf.
This completes the proof.
\end{proof}

\begin{lemma}\label{lem:open-nK_X}
If X is a normal scheme, and 
$$
U = \{ x \in X \mid \text{$X$ is $\Q$-Gorenstein at $x$}\}
$$
is the $\Q$-Gorenstein locus of $X$, then $U$ is open in $X$, 
and there is a positive integer $s$ such that 
$sK_X$ is Cartier on $U$. 
\end{lemma}

\begin{proof}
For every positive integer $m$, the set
$$
U_m = \{ x \in X \mid \text{$X$ is $m$-Gorenstein at $x$}\}
$$
is open in X (it is nonempty, since it contains $X_{\rm reg}$). 
Note that $U = \bigcup_{m \ge 1} U_m$. 
Furthermore, we have $U_k \subseteq U_m$ if $k$ divides $m$. 
It follows by the Noetherian property that there is a unique maximal set among all these
open sets. In other words, there is a positive integer $s$ such that
$U = U_s$. 
\end{proof}

\begin{lemma}\label{lem:an}
Let $X$ be a normal scheme, and let $g\colon Y \to X$ be a resolution of singularities. 
For a positive integer $m$, the divisor $mK_X$ is Cartier at a closed point $x \in X$
if and only if there is an open neighborhood $V$ of $x$ and
a $g$-exceptional divisor $E$ on Y such that 
$\O(mK_Y) \cong \O(E)$ on $g^{-1}(V)$. 
Furthermore, if the ground field is $\C$, 
then it is enough to find an open neighborhood $V$ of $x$ 
in the analytic topology such that 
$\O(mK_Y)^{\rm an} \cong \O(E)^{\rm an}$ on $g^{-1}(V)$.
\end{lemma}

\begin{proof}
Note that after fixing the Cartier divisor $K_Y$ on $Y$, we may take $K_X = g_*K_Y$. 
If $mK_X$ is Cartier, then $mK_Y - g^*(mK_X)$ is an integral
exceptional divisor. 
Thus, given $x \in X$ such that $mK_X$ is Cartier at $x$, it is enough to take an open 
neighborhood $V$ of $x$ where $mK_X$ is principal. 
Conversely, if there is $E$ as in the statement, then 
taking the push-forward and observing that
$g_*(mK_Y)=mK_X$ and $g_*E =0$, we see that $mK_X$ is linearly equivalent to zero
in a neighborhood of $x$.

Suppose now that $X$ is a complex variety, and assume that 
$\O(mK_Y)^{\rm an} \cong \O(E)^{\rm an}$ on $g^{-1}(V)$, 
where $V$ is an open neighborhood of $x$ in the analytic topology. 
It follows that there is a meromorphic function $\phi$ on $Y$ such that 
${\rm div}_Y(\phi) = mK_Y - E$ on $g^{-1}(V)$. In this case 
${\rm div}_X(\phi) = mK_X$ on $V$. Therefore $\O(mK_X)^{\rm an}$ is locally free of rank 
one at  $x$. Since $\O(mK_X)$ and $\O(mK_X)^{\rm an}$ 
have isomorphic completions at $x$, it follows by 
Nakayama's Lemma that $\O(mK_X)$ is locally free of rank one at $x$, hence $mK_X$ is Cartier
at this point.
\end{proof}

We are now ready to prove the key result of this appendix. 

\begin{proof}[Proof of Theorem~\ref{thm:Q-Gor}]
In this proof we only consider the closed points of the schemes involved.
Let $g \colon Y \to X$ be a resolution of singularities whose exceptional locus is a divisor
with simple normal crossings.
By a theorem of Verdier \cite{Ver}, we can write $X$ as a \emph{finite} disjoint union
$X = \bigsqcup X^\alpha$, with each $X^\alpha$ an irreducible
locally closed subset of $X$, such that the restriction
$g^\alpha \colon Y^\alpha \to X^\alpha$ of $g$ to $Y_{\alpha}=g^{-1}(X_{\alpha})$ is topologically locally trivial.
Let $Z^\alpha := \ov{X^\alpha} \smallsetminus X^\alpha$ (the closure being taken inside $X$).
Note that each $Z^\alpha$ is a closed subset of $X$.

By Lemma~\ref{lem:open-nK_X}, 
there is a positive integer $s$ such that $X$ is $\Q$-Gorenstein at a point $x$
if and only if $X$ is $s$-Gorenstein at $x$. 
By generic smoothness, generic flatness, and Lemma~\ref{lem:rest-O(mK_X)}, 
after possibly replacing
$T$ by a nonempty open subset, we can assume that the following properties hold:
\begin{enumerate}
\item[(1)]
$T$ is smooth;
\item[(2)]
$Y \to T$ is smooth,
the exceptional locus of $g$ has relative simple normal crossings 
over $T$,
 and
for every point $t \in T$, the induced morphism $g_t \colon Y_t \to X_t$ is
a resolution of singularities
 and every $g_t$-exceptional divisor is the restriction
to $Y_t$ of a $g$-exceptional divisor;
\item[(3)]
$X$ is flat over $T$, and both $\ov{X^\alpha}$ and $Z^\alpha$ are flat over $T$ 
for every $\alpha$;
\item[(4)]
For every $t \in T$, there is a canonical isomorphism
$\O(sK_X)\vert_{X_t} \cong \O(sK_{X_t})$.
\end{enumerate}
We will show that after this reduction the conclusion of the
theorem holds for every $t \in T$.

Fix any $t \in T$, and suppose that $x$ is a point where $X_t$ has
rational singularities. Since $f$ is flat and $T$ is smooth,
this implies that $X$ has rational singularities at $x$, and hence
in a neighborhood of $x$ (cf. \cite[Th\'eor\`eme~2 and Th\'eor\`eme~4]{Elk}). 

By Lemma~\ref{lem:open-nK_X}, we see that the conditions~(c) and~(d) are equivalent.
Furthermore, condition~(4) implies that~(b) and~(d) are equivalent,
and clearly~(b) implies~(a).
Therefore, in order to conclude it suffices to show that~(a) implies~(c).

We thus assume that~(a) holds, that is, that there is a positive
integer $m$ such that $mK_{X_t}$ is Cartier at $x$.
We denote by $X^\alpha_t$, $(\overline{X_{\alpha}})_t$, and $Z^\alpha_t$ the fibers of $X^\alpha$,
$\overline{X_{\alpha}}$, and $Z^\alpha$ over $t$.
Let $\cA := \{ \alpha \mid x \in \ov{X^\alpha_t}\}$. Note that
$$
x \in {\rm Int}\Big( \bigsqcup_{\alpha \in \cA} X^\alpha_t \Big).
$$
Indeed, if this were false, then for every open neighborhood $V$ of $x$ in $X_t$
we could find an $\alpha \not\in \cA$ such that $V \cap X^\alpha_t \ne \emptyset$.
By considering a nested sequence of open neighborhoods of $x$, we see that we can pick
$\alpha$ independent of $V$. As this holds for every $V$, we conclude that
$x \in \ov{X^\alpha_t}$, which contradicts the definition of $\cA$.

\noindent{\bf Claim}.
We have
$$
x \in {\rm Int}\Big( \bigsqcup_{\alpha \in \cA} X^\alpha \Big).
$$

\begin{proof}[Proof of claim.]
We argue by contradiction. Let us assume that $x$ is not in the interior
of $\bigsqcup_{\alpha \in \cA} X^\alpha$. Arguing as above,
we conclude that there is an $\alpha\not\in \cA$ such that $x \in \ov{X^\alpha}$.
Consider the morphism $\ov{X^\alpha} \to T$. We have $x \in (\ov{X^\alpha})_t$,
and since $x \not\in \ov{X^\alpha_t}$, 
there is an open neighborhood $V$ of $x$ in  $(\ov{X^\alpha})_t$
that is disjoint from $X^\alpha_t$, and hence from $X^\alpha$.
Therefore $V$ is contained in $Z^\alpha$, hence in $Z^\alpha_t$.
The closure of $V$ in $(\ov{X^\alpha})_t$,
and hence the closure of $Z^\alpha_t$ in $(\ov{X^\alpha})_t$, contains some irreducible 
component $W$ of $(\ov{X^\alpha})_t$. 
Since both $Z^{\alpha}$ and $\overline{X^{\alpha}}$ are flat over $T$, it follows that
if $x'\in W$ is a general (closed) point, then
$$\dim(W)=\dim(\cO_{Z^{\alpha},x'})-\dim(T)=\dim(\cO_{\overline{X^{\alpha}},x'})-\dim(T)$$
(see \cite[Proposition~III.9.5]{Har}).
The fact that $\dim(\cO_{Z^{\alpha},x'})=\dim(\cO_{\overline{X^{\alpha}},x'})$ contradicts the fact 
that $Z^{\alpha}$ is a proper closed subset of the irreducible set $\overline{X^{\alpha}}$,
and thus completes the proof of the claim. 
\end{proof}

We now fix a small contractible analytic open neighborhood 
$V \subseteq X$ of $x$ fully contained in $\bigsqcup_{\alpha \in \cA} X^\alpha$,
and such that $H^1(V,\cO_V^{\rm an})=0$.
Let $V_t = V \cap X_t$. 
We may and will assume that $V$ has rational singularities.
Furthermore, we may assume that $V_t$ is contained in any given
neighborhood of $x$, hence 
by Lemma~\ref{lem:an} and
the fact that $mK_{X_t}$ is Cartier at $x$ we may assume that there is a $g_t$-exceptional divisor
$E_t$ on $Y_t$ such that 
\begin{equation}\label{last}
\O(mK_{Y_t})^{\rm an} \cong \O(E_t)^{\rm an}
\quad\text{on}\ g_t^{-1}(V_t).
\end{equation}
It follows from condition~(2) that $mK_{Y_t}$ is the restriction of $mK_Y$
to $Y_t$, and $E_t$ is the restriction of a $g$-exceptional
divisor $E$ on $Y$. By Lemma~\ref{lem:an}, in order
to conclude the proof of the theorem, it suffices to show that there is
$\ell\geq 1$ such that 
$L^{\ell}$ is trivial, where 
$L=\O(mK_Y-E)^{\rm an}\vert_{g^{-1}(V)}$. 
Let $\gamma=c^1(L)\in H^2(g^{-1}(V),\ZZ)$. For every $x\in V$, we denote by
$\gamma_x$ the image of $\gamma$ in $H^2(Y_x,\Z)$ via the map induced by
$g^{-1}(x)=Y_x\hookrightarrow g^{-1}(V)$. 

Arguing by contradiction, let us assume that
$L^{\ell}$ is nontrivial for all $\ell \ge 1$.
It follows from \cite[(12.1.4)]{KM2} and the proof therein that in this case
we can find a $g$-exceptional curve $C \subset Y$,
with image $p := g(C)$ in $V$, such that $(L\cdot C)\neq 0$.
In particular, $\gamma_p\neq 0$.

By our choice of $V$, we have $p \in X^\alpha$ for some $\alpha \in \cA$.
Note that $V_t \cap X^\alpha_t \ne \emptyset$ by the definition of $\cA$
(recall that by \cite[Proposition~5]{GAGA},
the analytic closure of $X_t^\alpha$ in $X_t$ agrees with
the Zariski closure $\ov{X_t^\alpha}$).
Pick any point $q \in V_t \cap X^\alpha_t$. Since $X^\alpha$ is connected, and hence
path connected, we can fix a path $w \colon [0,1] \to X^\alpha$ joining $p$ to $q$.
As $g^\alpha$ is topologically locally trivial, moving along the path $w$
induces an isomorphism $H^2(Y_p,\Z) \cong H^2(Y_q,\Z)$.
Note that $\gamma_p$ is mapped to $\gamma_q$ via this isomorphism, hence
$\gamma_q\neq 0$.

On the other hand, (\ref{last}) implies that $L\vert_{g^{-1}(V)\cap Y_t}$
is trivial, hence so is $L\vert_{Y_q}$. Therefore $\gamma_q=0$, a contradiction.
 This completes the proof of the theorem.
\end{proof}

\begin{remark}\label{rem_s}
It follows from the above proof that in Theorem~\ref{thm:Q-Gor} one can take any  $s$ as given
by Lemma~\ref{lem:open-nK_X}, that is, such that $sK_X$ is Cartier on the largest open subset of
$X$ that is normal and $\QQ$-Cartier.
\end{remark}

We will need the following version of the result, which holds
over any algebraically closed field $k$ of characteristic zero. 

\begin{theorem}\label{Gor_index}
Let $f\colon X\to T$ be a morphism of schemes of finite type over $k$, with $T$ integral,
and let $\sigma\colon T\to X$ be a section of $f$, i.e. $f\circ \sigma=1_T$. Suppose that we have 
a countable dense subset $T_0\subseteq T$ of closed points such that for all $t\in T_0$,
at $\sigma(t)$ the fiber 
$X_t:=f^{-1}(t)$ is $\QQ$-Gorenstein and has rational singularities.
Then there is a nonempty open subset $U$ of $T$, and a positive integer $s$ such that
\begin{enumerate}
\item[i)] $X$ is normal and $sK_X$ is Cartier in a neighborhood of $\sigma(U)$.
\item[ii)] For every ${\rm (}$not necessarily closed${\rm )}$ point  $\xi\in U$, the fiber $X_{\xi}$ is normal at 
$\sigma(\xi)$, and we have
a canonical isomorphism 
$$\cO(sK_X)\vert_{X_{\xi}}\simeq\cO(sK_{X_{\xi}})$$
in a neighborhood of $\sigma(\xi)$. In particular, $sK_{X_{\xi}}$ is Cartier at $\sigma(\xi)$.
\end{enumerate}
\end{theorem}

\begin{proof}
It is clear that we may replace $T$ by any nonempty open subset $V$
(note that the set $T_0\cap V$ is countable and dense in $V$). Furthermore,
if $W\subseteq X$ is an open subset such that $\sigma^{-1}(W)$ is nonempty, it is enough 
to prove the theorem for $W\cap f^{-1}(\sigma^{-1}(W))\to \sigma^{-1}(W)$ (note that $\sigma$ induces a section of this morphism). 

After replacing $T$ by an open smooth subset, we may assume that $T$ is smooth, and
$f$ is flat (we again use generic flatness). By \cite[Th\'eor\`eme~4]{Elk}, there is an open
subset
$W_1$ of $X$ whose closed points $x\in W_1$ are precisely those such that $X_{f(x)}$ has rational singularities at $x$. Since $\sigma(t)\in W_1$ for every $t\in T_0$, we see that $\sigma^{-1}(W_1)$ is nonempty. 
After replacing $X$ by $W_1\cap f^{-1}(\sigma^{-1}(W_1))$, we may assume that
every closed fiber $X_t$ has rational singularities. In this case, by \cite[Th\'eor\`eme~2]{Elk}
$X$ has rational singularities. Furthermore, all fibers of $f$ are normal.

We apply Lemma~\ref{lem:open-nK_X} to get the open subset $W\subseteq X$
consisting of the points in $X$ where $K_X$ is $\QQ$-Cartier. Let $s$ be such that
$sK_X$ is Cartier on $W$. In order to prove the theorem it is enough to show that
$\sigma^{-1}(W)$ is nonempty. Indeed, if this is the case we may replace 
$X$ by $W\cap f^{-1}(\sigma^{-1}(W)$, in which case $sK_X$ is Cartier.
After replacing $T$ by a nonempty open subset, we may assume by
 Lemma~\ref{lem:rest-O(mK_X)} that $\cO(sK_X)\vert_{X_{\xi}}\simeq\cO(sK_{X_{\xi}})$
 for every $\xi\in T$. This would prove the theorem. 
 
 If the ground field is $\CC$, then by Theorem~\ref{thm:Q-Gor} (see also Remark~\ref{rem_s})
 we may replace $T$ by an open subset and assume that for every closed point $t\in T$, 
the divisor $sK_{X_t}$ is Cartier at a closed point $x$ if and only if $x\in W$. 
 In this case we see that $T_0\subseteq \sigma^{-1}(W)$, hence $\sigma^{-1}(W)$ is nonempty. 

For an arbitrary $k$, we can find a subfield $k'$ of $k$ 
of countable transcendence degree over $\QQ$, such that there are morphisms $f'\colon X'\to T'$
and $\sigma'\colon T'\to X'$
of schemes over $k'$, 
 with $f$ and $\sigma$ obtained from $f'$, respectively $\sigma'$, by base-extension via
$\Spec(k)\to\Spec(k')$, and such that the points in $T_0$ are defined over $k'$.
If $\phi\colon T\to T'$ is the natural morphism, it follows that 
$T'_0:=\phi(T_0)$ consists of $k'$-rational closed points.

There is an embedding $k'\hookrightarrow\CC$. Let 
$\widetilde{f}\colon \widetilde{X}\to\widetilde{T}$ and $\widetilde{\sigma}\colon
\widetilde{T}\to\widetilde{X}$
be the morphisms
obtained from $f'$, respectively $\sigma'$,
by base-extension to $\CC$. If $\psi\colon \widetilde{T}\to T'$
is the natural morphism, we choose (closed) points $\widetilde{t}\in \psi^{-1}(t')$
for all $t'\in T'_0$. 
Let $\widetilde{T}_0$ be the set consisting of these closed points. 
It follows from Lemma~\ref{alg_closed} i)
that for every $\widetilde{t}\in\widetilde{T_0}$, the fiber $\widetilde{X}_{\widetilde{t}}$ is $\QQ$-Gorenstein and with rational singularities at
 $\widetilde{\sigma}(\widetilde{t})$ 
 (the assertion about rational singularities follows easily from definition 
 by considering base-extensions of resolutions of singularities).
 
 On the other hand, if $W'\subseteq X'$ and $\widetilde{W}\subseteq\widetilde{X}$ are
 the subsets where $X'$ and $\widetilde{X}$, respectively, are normal and $\QQ$-Gorenstein,
 then by Lemma~\ref{alg_closed} i) we see that $\widetilde{W}=W'
 \times_{\Spec(k')}\Spec(\CC)$ and $W=W'\times_{\Spec(k')}\Spec(k)$. 
 Furthermore, the divisors $sK_{X'}$ and $sK_{\widetilde{X}}$ are Cartier on 
 $W'$ and $\widetilde{W}$, respectively.
 We have already seen
 that $\widetilde{\sigma}^{-1}(\widetilde{W})$ is a nonempty open subset of $\widetilde{T}$.
 The closure   of $\psi(\widetilde{T}\smallsetminus\widetilde{\sigma}^{-1}(\widetilde{W}))$ is a proper closed subset of $T'$. If $t'$ 
 is a closed point in the complement of this closed set, and if $t\in\phi^{-1}(t')$, then $\sigma(t)\in W$. 
 Therefore $\sigma^{-1}(W)$ is nonempty, and this completes the proof of the theorem.
\end{proof}

In order to state the next corollary, we introduce some notation.
Let $f\colon X\to T$ be a morphism of schemes of finite type over $k$, with $T$ integral, and 
$\sigma\colon T\to X$ a section of $f$.
Suppose that $(t_m)_{m\geq 1}$ is a dense sequence of closed points in $T$ such that
each $X_{t_m}$ is klt around $\sigma(t_m)$. 
Suppose that $\frA=\prod_{j=1}^r\fra_j^{p_j}$ is an $\RR$-ideal on $X$, such that
each $\fra_j$ vanishes along $\sigma(T)$, but it does not vanish along any fiber of $f$. For every
(not necessarily closed) point $\xi\in T$, we put $\fra_{j, \xi}=\fra_j\cdot\cO_{X_{\xi}}$
and $\frA_{\xi}=\prod_j\fra_{j, \xi}^{p_j}$.
We also denote by $\frm_{\xi}$ the ideal defining 
$\sigma(\xi)$ in $X_{\xi}$.

\begin{corollary}\label{cor_for_key_technical}
With the above notation and assumptions, the following hold:
\begin{enumerate}
\item[i)] If $\eta$ is the generic point of $T$, then $X_{\eta}$ is klt at $\sigma(\eta)$. Furthermore, there is
a subset $J$ of $\ZZ_{>0}$ such that $\{t_i\mid i\in J\}$ is dense in $T$, and 
for every $i$ in $J$
$$\lct_{\sigma(\eta)}(X_{\eta},\frA_{\eta})=\lct_{\sigma(t_i)}(X_{t_i},\frA_{t_i}).$$
\item[ii)] If $E$ is a divisor over $X_{\eta}$ computing  $\lct_{\sigma(\eta)}(X_{\eta},\frA_{\eta})$,
then after possibly replacing $J$ by a smaller subset $J_1$ with the same properties, we may assume that, in addition, for 
every $i\in J_1$ we have a divisor $E_i$ over $X_{t_i}$ that computes
$\lct_{\sigma(t_i)}(X_{t_i},\frA_{t_i})$, and such that 
$\ord_E(\frm_{\eta})=\ord_{E_i}(\frm_{t_i})$ and 
$\ord_E(\fra_{j, \eta})=\ord_{E_i}(\fra_{j, t_i})$ for all $j\leq r$.
In particular, if $E$ has center equal to $\sigma(\eta)$, then
each $E_i$ with $i\in J_1$ has center 
$\sigma(t_i)$.
\end{enumerate}
\end{corollary}

\begin{proof}
It is clear that we are allowed to replace $T$ by any open subset $V$, in which case
we need to replace $\ZZ_{>0}$ by $J=\{i\mid t_i\in V\}$.
Since klt varieties have rational singularities (see \cite[Corollary~11.14]{Kol2}),
we may apply Theorem~\ref{Gor_index} to $f$. Let $U\subseteq T$ and $s$ be given by this theorem. After replacing $T$ by $U$, we may
assume $U=T$. Since $sK_X$ is Cartier around $\sigma(T)$, and since we are only interested
in the behavior around $\sigma(T)$, we may replace $X$ by the open subset where 
$sK_X$ is Cartier, and therefore assume $sK_X$ is Cartier.

Consider now a log resolution $h\colon Y\to X$ of $(X,\frA)$. Let ${\mathcal E}$ be the simple normal
crossings divisor on $Y$ given by the sum of the $h$-exceptional divisors and of the divisors 
appearing in the supports of the ideals $\fra_j\cO_Y$. By generic smoothness, after possibly
replacing $T$ by an open subset,
we may assume the following properties: 
\begin{enumerate}
\item[(1)] The composition $f\circ h$ is smooth, and 
${\mathcal E}$ has relative simple normal crossings over $T$. 
\item[(2)] For every prime divisor $F_j$ in ${\mathcal E}$, its image $Z_j$ in $X$
is flat over $T$ and maps onto $T$. 
\item[(3)] Furthermore, we may and will assume that each $Z_j$ contains
$\sigma(T)$ (otherwise we simply replace $T$ by $T\smallsetminus \sigma^{-1}(Z_j)$).
\end{enumerate}
It follows from (1) that
for every (not necessarily closed) point $\xi\in U$, the morphism $h_{\xi}\colon Y_{\xi}\to X_{\xi}$ is a log resolution of $(X_{\xi},\frA_{\xi})$.
By (2), if  $F$ is a component of ${\mathcal E}$ that is $h$-exceptional, then
$F_{\xi}$ is $h_{\xi}$-exceptional (note that $F_{\xi}$ is smooth, but might not be connected).
Since $\cO(sK_X)\vert_{X_{\xi}}\simeq\cO(sK_{\xi})$ in a neighborhood of $\sigma(\xi)$,
we deduce that $K_{Y/X}\vert_{Y_{\xi}}=K_{Y_{\xi}/X_{\xi}}$ over the inverse image of an open neighborhood of $\sigma(\xi)$. 
Since $\sigma(\xi)\in h_{\xi}(F_{\xi})$ for every prime divisor $F$ in ${\mathcal E}$, it follows that
each $X_{\xi}$ is klt and $\lct_{\sigma(\xi)}(X_{\xi},\frA_{\xi})=\lct(X,\frA)$.
Applying this with $\xi=t_i$ and $\xi=\eta$ gives the assertions in i).

Suppose now that $E$ is as in ii). $E$ appears as a prime divisor on some log resolution of
$(X_{\eta},\frA_{\eta}\cdot\frm_{\eta})$. Since this log resolution is defined over $k(\eta)$, it follows that after replacing 
$T$ by a suitable open subset, we may assume that this log  resolution is equal to $h_{\eta}$ for some log resolution $h\colon Y\to X$ as above. In fact, we may assume that $h$ is a log resolution of  $(X,\frA\cdot \fra_{\sigma(T)})$, where 
$\fra_{\sigma(T)}$ is the ideal defining $\sigma(T)$ in $X$. We may again assume that $h$
satisfies (1)-(3) above. There is a prime divisor $F$ in ${\mathcal E}$ such that 
$E=F_{\eta}$. It is then clear that, for every $i$, we may take $E_i$ to be any connected component
of $E_{t_i}$, and that the divisors $E_i$ satisfy ii).
\end{proof}

\begin{remark}\label{independence_of_exponents}
It follows from the proof of the corollary that the set $J$ in i) can be chosen independently 
of the exponents $p_1,\ldots,p_r$. In fact, while the convention for $\RR$-ideals is that all
exponents are positive, it is clear that the result in the corollary still holds if some (but not all)
of the $p_i$ are allowed to be zero.
\end{remark}

\frenchspacing

\providecommand{\bysame}{\leavevmode \hbox \o3em
{\hrulefill}\thinspace}

\end{document}